\documentclass[10pt]{amsart}

\usepackage{amsmath}
\usepackage{amssymb}

\newtheorem{prop}{Proposition}[section]
\newtheorem{theorem}[prop]{Theorem}
\newtheorem{cor}[prop]{Corollary}
\newtheorem{lemma}[prop]{Lemma}

\theoremstyle{definition}
\newtheorem{Def}[prop]{Definition}

\newtheorem{remark}[prop]{Remark}

\numberwithin{equation}{section}

\newcommand{\R}{\mathbb R}

\renewcommand{\leq}{\leqslant}

\def\Dx{\Delta_x}
\def\Cal{\mathcal}
\def\({\left(}
\def\){\right)}
\def\Nx{\nabla_x}
\def\Dt{\partial_t}
\def\dist{\operatorname{dist}}

\def\eb{\varepsilon}
\def\dist{\operatorname{dist}}
\def\ext{\operatorname{Ext}}
\def\Bbb{\mathbb}
\begin{document}

\title[Quintic wave equation]{Smooth attractors for the quintic wave equations with fractional damping}
\author[] {Anton Savostianov and Sergey Zelik}

\begin{abstract}  Dissipative wave equations with critical quintic nonlinearity and damping term involving the fractional Laplacian are considered. The  additional regularity of  energy solutions is established by constructing  the new  Lyapunov-type functional  and based on this, the global well-posedness  and dissipativity of the energy solutions as well as the existence of a smooth global and exponential attractors of finite Hausdorff and fractal dimension is verified.
\end{abstract}

\subjclass[2000]{35B40, 35B45}
\keywords{damped wave equation, fractional damping, critical nonlinearity, global attractor, smoothness}
\thanks{
This work is partially supported by the Russian Ministry of Education and Science (contract no.
8502).}
\address{University of Surrey, Department of Mathematics, \newline
Guildford, GU2 7XH, United Kingdom.}

\email{s.zelik@surrey.ac.uk}
\email{a.savostianov@surrey.ac.uk}

\maketitle
\tableofcontents
\section{Introduction}
Wave equations with the damping term involving the fractional Laplacian
\begin{equation}\label{0.eqmain}
\Dt^2 u+\gamma(-\Dx)^\theta\Dt u-\Dx u+f(u)=0,\  \ x\in\Omega\subset\R^n,
\end{equation}
where $\gamma>0$, $\theta\in[0,2]$ and $f(u)$ is a nonlinear interaction function, are of big permanent interest. These equations model various oscillatory processes in a lossy media including the nonlinear elasticity, electrodynamics, quantum mechanics, etc. We mention here also  the relatively recent applications to the acoustic waves propagation in viscous/viscoelastic  media. Then, the coefficient $\theta$ in the non-local term of the equation is related with the power law for the dependence of the acoustic attenuation $a(\omega)$ on the angular frequency $\omega$, namely,
\begin{equation}
a(\omega)=a_0|\omega|^{2\theta},
\end{equation}
see \cite{tree, chen} and references therein. The classical choices of the exponent $\theta$ are $\theta=0$ and
$\theta=1$ which correspond to the usual damped wave equations and the so-called strongly damped wave equations respectively although an increasing attention is attracted for the general case $\theta\in(0,2)$ as well, see for instance \cite{pri} and the literature cited therein.
\par
Mathematical properties of equation \eqref{0.eqmain} and the related equations, including the well-posedness, regularity and the asymptotic behavior of solutions have been  studied in many papers, see \cite{ACH,BV,carc1,carc2,carc3,tri1,tri2,bk_ChuLas2010,Chu2010,GM,K86,KZwvEq2009,mossad,MZDafer2008,PataZel2006 rem,PataZel2006,TemamDS,We} and references therein. These properties depend strongly on the value of $\theta$ and on the growth rate of the nonlinearity $f$. For instance, equation \eqref{0.eqmain} is {\it hyperbolic} when $\theta=0$ and (under natural assumptions on $f$, e.g., for $f=0$) and is solvable forward and backward in time, so the solution operators $S(t)$ generate a $C_0$-group in the proper phase space and do not possess any smoothing property on finite time, only the asymptotic smoothing as $t\to\infty$ holds, see e.g., \cite{TemamDS}.
\par
 In the case, $0<\theta<\frac12$, equation \eqref{0.eqmain} generates the so-called $C^\infty$-semigroup, so it is already solvable only forward in time and possesses the smoothing property on finite time (even starting from the non-smooth initial data, the solution becomes smooth for $t>0$). However, this smoothing is not strong enough to give the maximal regularity and the solution semigroup is not {\it analytic}, see \cite{tri1}. In contrast to that, the solution semigroup is {\it analytic} for $\frac12\le\theta<1$ and equation \eqref{0.eqmain} is actually {\it parabolic}, see \cite{tri2}. For instance, in the borderline case $\theta=1/2$ and $\gamma=2$ the change of variable $v=\Dt u+(-\Dx)^{1/2}u$ transforms \eqref{0.eqmain} to the following system:
\begin{equation}\label{0.par}
\begin{cases}
\Dt u+(-\Dx)^{1/2}u=v,\\
\Dt v+(-\Dx)^{1/2}v=f(u)
\end{cases}
\end{equation}
and the parabolicity is obvious. For general $\gamma\ne2$, the change of variable is $v=\Dt u+\beta(-\Dx)^{1/2}u$ where $\beta^2-\gamma\beta+1=0$, so $\beta$ will be complex if $\gamma<2$ which does not destroy the parabolicity, but just reflects the fact that the eigenvalues of the problems are complex conjugate if $\gamma<2$.
\par
Finally, when $\theta=1$, the solution semigroup remains analytic, but nevertheless equation \eqref{0.eqmain} is {\it not parabolic} and  possesses
only partial smoothing on a finite time, see \cite{KZwvEq2009,PataZel2006} for more details.
\par
Let us now discuss the dependence of \eqref{0.eqmain} on the growth rate of the nonlinearity $f$. For simplicity, we restrict ourselves to consider only the case where $\Omega$ is a bounded domain in $\R^3$ with proper boundary conditions and the nonlinearity $f$ satisfies the following natural assumption:
\begin{equation}\label{0.f}
-C+\kappa|u|^q\le f'(u)\le C(1+|u|^q), \ \ q>0,
\end{equation}
where $\kappa,C>0$ are some constants.
For instance, polynomials of odd degree and positive leading coefficient clearly satisfy this assumption. Thus, the growth rate of $f(u)$ when
$u\to\infty$ is determined by the growth exponent $q$. Note that \eqref{0.eqmain} possesses the following {\it energy} equality at least on a formal level:
\begin{equation}\label{0.energy}
\frac d{dt}\(\frac12\|\Dt u\|^2_{L^2(\Omega)}+\frac12\|\Nx u\|^2_{L^2(\Omega)}+(F(u),1)\)=-\gamma\|(-\Dx)^{\theta/2}\Dt u\|^2_{L^2(\Omega)},
\end{equation}
where $F(s)=\int_0^sf(v)\,dv$ is the potential of the non-linearity $f$ and $(u,v)$ stands for the scalar product of functions $u$ and $v$ in $H:=L^2(\Omega)$. By this reason, the {\it energy} solutions are defined in such way that all terms in that inequality have sense (at least after  integration in time), namely, for the case of Dirichlet boundary conditions, the following regularity of energy solutions is assumed:
\begin{multline}\label{0.ensol}
(u(t),\Dt u(t))\in\Cal E:=[H^1_0(\Omega)\cap L^{q+2}(\Omega)]\times L^2(\Omega),\\ (-\Dx)^{\theta/2}\Dt u\in L^2([0,T],L^2(\Omega)).
\end{multline}
Here and below, $H^s(\Omega)$ denotes the Sobolev space of distributions whose derivatives up to order $s$ belong to $L^2(\Omega)$ and $H_0^s$ is the closure of $C^\infty_0(\Omega)$ in the metric of $H^s(\Omega)$.
\par
 As usual, for sufficiently small $q$ which is less than the {\it critical} exponent $q_*=q_*(\theta)$ (the so-called {\it sub-critical} case), the non-linearity is subordinated to the linear part of the equation, so the analytic properties of it remains similar to the linear case $f=0$. In particular, we have the existence, uniqueness and dissipativity of the energy solutions in an almost straightforward way. However, nothing similar works in the {\it super-critical} case $q>q_*(\theta)$ where usually only the existence of a weak solution without the uniqueness and further regularity is known. As expected, equations with super-critical non-linearities may generate singularities in finite time even staring with smooth initial data (similarly, e.g., the so-called self-focusing phenomenon in the  non-linear Schr\"odingier equation or the gravitational collapse in general relativity) and this, in particular, may cause also the non-uniqueness. Unfortunately, to the best of our knowledge, no such examples are known for equations \eqref{0.eqmain} with dissipative non-linearities $f$ satisfying \eqref{0.f}, so being pedantic, the precise value of the critical exponent $q_*(\theta)$ is not known. Moreover, it is a priori not excluded that smooth solutions of the concrete
equation \eqref{0.eqmain} do not blow up in finite time, for instance, due to the presence of some non-trivial extra Lyapunov type functionals controlling the higher norms of the solution, then $q_*(\theta)=\infty$ and energy solutions of \eqref{0.eqmain} are well-posed and dissipative no matter how fast is the growth of the non-linearity $f$.
\par
As shown in \cite{KZwvEq2009}, see also \cite{PataZel2006} that is indeed the case when $\theta=1$, so the blow up of smooth solutions is impossible and the uniqueness theorem holds for the energy solutions which are asymptotically smooth as $t\to\infty$ no matter how fast the non-linearity $f$ is growing. It is also indicated  there that
\begin{equation}
q_*(\theta)=\infty,\ \ \theta\in[\frac34,1],
\end{equation}
so the critical growth exponent does not exist for $\theta\ge\frac34$ as well. In the case $\theta\in(\frac12,\frac34)$, the result of \cite{KZwvEq2009} gives the estimate
\begin{equation}
q_*(\theta)\ge \frac{8\theta}{3-4\theta},\ \ \theta\in(\frac12,\frac34).
\end{equation}
The non-parabolic case $\theta\in[0,\frac12)$ is more delicate. In the purely hyperbolic case $\theta=0$ the assumption $q\le2$ (at most cubic growth rate of the nonlinearity is allowed) is posed in the most part of works in bounded domains, see \cite{BV,TemamDS,Zel2004} and references therein, although the quintic growth rate ($q_*(0)=4$) is expected here as critical. This conjecture  is partially verified in \cite{stri} (based on the recent progress in Strihartz-type estimates for bounded domains), where the global well-posedness of weak solutions has been proved for the case of quintic hyperbolic equation. To the best of our knowledge, the corresponding attractor theory for that equation in bounded domains is not developed yet.
\par
In the other borderline case $\theta=\frac12$, it is also expected that the critical exponent $q_*(1/2)=4$, so the natural conjecture here is
$$
q_*(\theta)=4,\ \ \theta\in[0,\frac12].
$$
The sub-critical case $q<4$ when $\theta=1/2$ is completely understood now-a-days, see \cite{carc1,carc2,carc3} and references therein. Some results on the well-posedness of this equation in the critical quintic case $q=4$ are also obtained there based on the so-called mild solutions. However, to the best of our knowledge, the analog of the so-called non-concentration effect which is typical for the quintic wave equations (with $\theta=0$), see \cite{stri,noncon}, is not known in the case when $\theta=\frac12$, so these results are not very helpful for proving the absence of blow up of smooth solutions and/or building up a reasonable attractor theory. Thus, clarifying the situation with the quintic nonlinearity and $\theta=\frac12$ was a long-standing open problem.
\par
The aim of this paper is to give a solution of  this problem. Namely, we will show that the energy solution of \eqref{0.eqmain} in 3D bounded domain with $\theta=\frac12$ and quintic non-linearity is unique and is smooth (say, $u(t)\in H^2(\Omega)$) for all $t>0$. Moreover, the associated semigroup in the energy phase space $\Cal E=H^1_0(\Omega)\times L^2(\Omega)$ is dissipative and possesses a smooth global attractor of finite Hausdorff and fractal dimension. To prove this result, we show that {\it any} energy solution possesses the following  extra regularity:
\begin{equation}\label{0.extra}
u\in L^2([0,T],H^{3/2}(\Omega))
\end{equation}
which is enough to verify the above stated properties in a more or less standard way.
\par
Note that the regularity \eqref{0.extra} does not follow from the energy equality \eqref{0.energy} and, analogously to \cite{KZwvEq2009}, some extra
Lyapunov-type functionals are necessary for obtaining it. Namely, as not difficult to see that the desired estimated will be obtained if we would succeed to multiply equation \eqref{0.eqmain} on $(-\Dx)^{1/2}u$ (or which is the same, to multiply the second equation of \eqref{0.par} by $v$) and estimate the non-linear term with fractional Laplacian: $(f(u),(-\Dx)^{1/2}u)$. To estimate this term, we utilize the well-known formula
$$
\|u\|_{H^s(\Omega)}^2\sim \|u\|_{L^2(\Omega)}^2+\int_{x\in\Omega}\int_{y\in\Omega}\frac{|u(x)-u(y)|^2}{|x-y|^{3+2s}}\,dx\,dy,\ \ s\in(0,1),
$$
see, e.g., \cite{triebel}. In particular, for more simple case of {\it periodic} boundary conditions $\Omega=\Bbb T^3:=[0,2\pi]^3$, we establish the following representation:
\begin{multline}\label{0.good}
(f(u),(-\Dx)^su)=\\=c_s\int_{h\in\R^3}\int_{x\in\Bbb T^3}\frac{(f(u(x+h))-f(u(x)),u(x+h)-u(x))}{|h|^{3+2s}}\,dx\,dh,
\end{multline}
where $c_s>0$ is some positive constant which is independent of $f$ and $u$. This estimate together with the left inequality of \eqref{0.f} imply that
$$
(f(u),(-\Dx)^{1/2}u)\ge -C\|u\|_{H^{1/2}(\Omega)}^2
$$
and this is enough to verify \eqref{0.extra}, see Sections \ref{s1} for the details.
\par
The case of a general bounded domain $\Omega\subset\R^3$ endowed by the Dirichlet boundary conditions is a bit more difficult since we do not know the analogue of \eqref{0.good} for that case and have to proceed in a different way using the {\it odd} extension of the solution $u$ through the boundary. Then, the already obtained energy estimate occurs to be sufficient for estimating the extra terms arising under the extension, so with the help of this trick, we actually reduce the general case to the case of periodic boundary conditions considered before, see Section \ref{s6}.
\par
The paper is organized as follows.
\par
The definitions of functional spaces as well as assumptions on the non-linearity and external forces used throughout of the paper are given in Section \ref{s1}. Moreover, the key technical tool \eqref{0.good} is verified here for the case of periodic boundary conditions.
\par
The weak energy solutions are introduced and studied in Section \ref{s2}. In particular, extra regularity \eqref{0.extra} is verified here in the case of periodic boundary conditions. Moreover, as shown there, that extra regularity is enough to verify that any energy solution satisfies the energy {\it equality} \eqref{0.energy}.
\par
The well-posedness and smoothing property for the energy solutions are verified in Section \ref{s4} under the assumption \eqref{0.extra} of extra regularity. The existence of global and exponential attractors for the associated solution semigroup are proved in Section \ref{s5}.
\par
Finally, the case of Dirichlet boundary conditions is considered in Section \ref{s6}. In that section, we verify the extra regularity \eqref{0.extra} for the Dirichlet case  using the odd extension of a solution through boundary. Important properties of such extension operator are collected
in Appendix \ref{s7}.
\par
 To conclude, we note that the methods developed in the paper work not only for the borderline case $\theta=\frac12$ and not only for quintic nonlinearities. We give the further applications in the forthcoming paper.

 \section{Preliminaries}\label{s1}
 In that section, we briefly remind the properties of fractional Sobolev spaces related with our problem and verify the key formula \eqref{0.good}. In order to be able to consider the cases of Dirichlet and periodic boundary conditions from the unified point of view, we add the extra damping term $\alpha\Dt u$ to equation  \eqref{0.eqmain} and will consider the following problem:
\begin{equation}\label{OriginalSys}
\begin{cases}
\Dt^2u+\gamma(-\Dx)^\frac{1}{2} \Dt u+\alpha\Dt u-\Dx u+f(u)=g,\\
u\big|_{t=0}=u_0,\quad \Dt u\big|_{t=0}=u_1,
\end{cases}
\end{equation}
either in a bounded smooth domain $\Omega\subset\R^3$ with Dirichlet boundary conditions or on a torus $\Omega=\Bbb T^3:=[0,2\pi]^3$ with periodic boundary conditions. Here $\Dx$ is a Laplacian with respect to the variable $x$, $u=u(t,x)$ is an unknown function, $\gamma$ and $\alpha$ are fixed strictly positive numbers, $g\in L^2(\Omega)$ and $f\in C^1(\R)$ satisfies the following growth and dissipativity assumptions:
\begin{equation}
\label{1.f}
\begin{cases}
1.\ f(u)u\ge-C+\kappa |u|^2,\\
2.\ f'(u)\ge -K,\\
3.\ |f'(u)|\le C(1+|u|^4),
\end{cases}
\end{equation}
where $C$, $\kappa$, and $K$ are given  positive constants. Finally, the initial data $\xi_u(0):=(u_0,u_1)$ is assumed belonging to the energy space
$$
\Cal E:=H^1_0(\Omega)\times L^2(\Omega).
$$
Here and below $H^s(\Omega)$, $s\in\R$, stands for the classical Sobolev spaces of distributions whose derivatives up to order $s$ belong to $L^2(\Omega)$ and $H^s_0(\Omega)$ means the closure of $C^\infty_0(\Omega)$ in the metric of $H^s(\Omega)$. Recall that, for the non-integer positive values of $s$ the norm in the space $H^s(\Omega)$ is defined via the  interpolation:
\begin{equation}\label{1.frac}
\|u\|_{H^s}^2:=\|u\|_{H^{[s]}}^2+\int_{x,y\in\Omega}\frac{|D^{[s]}u(x)-D^{[s]}u(y)|^2}{|x-y|^{3+2(s-[s])}}\,dx\,dy,
\end{equation}
where $[s]$ is the integer part of $s$ and $D^{[s]}$ stands for the collection of all partial derivatives of order $[s]$. For negative $s$, the space $H^s(\Omega)$ is defined by the duality: $H^{s}(\Omega):=[H^{-s}_0(\Omega)]^*$, see e.g., \cite{triebel} for the details. The standard inner product in $L^2(\Omega)$ will be denoted by $(u,v)$.
\par
We now remind the fractional powers of the Laplacian. To this end, let $0\le\lambda_1\le\lambda_2\le \cdots$ be the eigenvalues of the Laplacian $-\Dx$ and $\{e_i\}_{i=1}^\infty\subset C^\infty(\Omega)$ be the associated eigenvectors. Then, due to the Parseval equality,
$$
\|u\|^2_{L^2(\Omega)}=\sum_{i=1}^\infty u_i^2,\ \ u_i:=(u,e_i),\ \ u=\sum_{i=1}^\infty u_i e_i.
$$
The scale of Hilbert spaces $H^s_{\Delta}$, $s\in\R$, associated with the Laplacian is defined as a completion of $L^2(\Omega)$ with respect to the following norm:
\begin{equation}\label{1.hsd}
\|u\|^2_{H^s_{\Delta}}:=\sum_{i=1}^\infty(1+\lambda_i)^s u_i^2,\ \  u:=\sum_{i=1}^\infty u_i e_i
\end{equation}
and the fractional Laplacian $(-\Dx)^\theta$, $\theta\ge0$, acts from $H^s_{\Delta}$ to $H^{s-2\theta}_{\Delta}$ via the following expression:
\begin{equation}\label{1.flap}
(-\Dx)^\theta u:=\sum_{i=1}^\infty\lambda_i^\theta u_i e_i,\ \ \ u=\sum_{i=1}^\infty u_ie_i.
\end{equation}
Recall also that, in the case of Dirichlet boundary conditions, $\lambda_1>0$ and therefore $(-\Dx)^{\theta}$ is an isomorphism between $H^s_{\Delta}$ and $H^{s-2\theta}_{\Delta}$ for all $\theta,s\in\R$. In the case of periodic boundary conditions $\lambda_1=0$, so the operator $(-\Dx)^{\theta}$ is not invertible and one should replace it by $(-\Dx+1)^{\theta}$ in order to restore the isomorphism.
\par
The relations between the Sobolev spaces $H^s(\Omega)$ and the spaces $H^s_{\Delta}$ in the case of Dirichlet boundary conditions are well-known at least in the case of smooth domains. Namely, for $s>0$ and $s\ne 1/2,5/2,9/2,\cdots$ one has
\begin{equation}\label{1.fgood}
H^s_\Delta=H^s(\Omega)\cap \{u\big|_{\partial\Omega}=-\Dx u\big|_{\partial\Omega}=\cdots=(-\Dx)^{[(s-1)/2]} u\big|_{\partial\Omega}=0\}
\end{equation}
and, in particular, by duality, $H^{s}_\Delta=H^s(\Omega)$ for $-3/2\le s\le0$, $s\ne-1/2$. In contrast to that, in the case of, say, $s=1/2$, the space $H^{1/2}_\Delta$ is a proper (dense) subset of $H^{1/2}(\Omega)=H^{1/2}_0(\Omega)$ determined by the following norm:
\begin{equation}\label{1.fbad}
\|u\|^2_{H^{1/2}_\Delta}\sim \|u\|^2_{H^{1/2}(\Omega)}+\int_\Omega d(x)^{-1}|u(x)|^2\,dx<\infty,
\end{equation}
where $d(x)$ is a distance from $x$ to the boundary $\partial\Omega$ and, by duality,
$$
H^{-1/2}(\Omega)\subset H^{-1/2}_\Delta\subset D'(\Omega),
$$
see \cite{LionMag} or \cite{triebel} for more details.
\par
We now consider the special case $\Omega=\Bbb T^3$ with periodic boundary conditions. In that case, since there is actually no boundary, we have the equality
$$
H^s(\Omega)=H^s_\Delta
$$
for all $s\in\R$. Moreover, the eigenvectors of the Laplacian $(-\Dx)$ are now the complex exponents $e^{ik.x}$,  $k\in\Bbb Z^3$ (here and below $a.b$ stands for the usual dot product in $\R^3$), with the corresponding eigenvalues $|k|^2$, so the Parseval equality  reads
\begin{equation}\label{1.parseval}
\|u\|^2_{L^2(\Omega)}=(2\pi)^3\sum_{k\in\Bbb Z^3}|u_k|^2,\ \ \ u(x)=\sum_{k\in\Bbb Z^3}u_ke^{ik.x},\ \ u_k:=\frac1{(2\pi)^3}(u,e^{ik.x})
\end{equation}
and also, for $s\ge0$,
\begin{equation}
\label{eq norm |u|+|Ds/2 u|}
d_1(\|u\|^2_{L^2}+\|(-\Delta)^\frac{s}{2} u\|^2_{L^2})\leq\|u\|^2_{H^s}\leq d_2(\|u\|^2_{L^2}+\|(-\Delta)^\frac{s}{2} u\|^2_{L^2}),
\end{equation}
where $d_1$ and $d_2$ are some positive constants depending on $s$.
\par
The next representation of the norm $\|(-\Dx)^{s/2}u\|^2_{L^2}$ is crucial for what follows.
\begin{lemma}\label{Lem1.nice}Let $s\in (0,1)$ and $u\in H^s(\Bbb T^3)$, then the following identity holds:
\begin{equation}
\label{1.eqnorms}
\|(-\Dx)^\frac{s}{2}u\|^2_{L^2}=c\int_{\mathbb{R}^3}\int_{\mathbb{T}^3}\frac{|u(x+h)-u(x)|^2}{|h|^{3+2s}}dxdh,
\end{equation}
for some strictly positive $c=c_s$ which depends only on $s$.
\end{lemma}
\begin{proof}
 Indeed, let $u\in H^s(\mathbb{T}^3)$, where $s\in(0,1)$ and
\begin{equation}
u(x)=\sum_{k\in\mathbb{Z}^3}u_ke^{ik.x}.
\end{equation}
Then by Parseval equality,
\begin{multline*}
\int_{\mathbb{T}^3}(u(x+h)-u(x))^2dx=(2\pi)^3\sum_{k\in\mathbb{Z}^3}|u_k|^2|e^{ik.h}-1|^2=\\=
(2\pi)^3\sum_{k\in\mathbb{Z}^3}4|c_k|^2\sin^2(k.h/2).
\end{multline*}
Consequently,
\begin{multline}
\int_{\mathbb{R}^3}\int_{\mathbb{T}^3}\frac{(u(x+h)-u(x))^2}{|h|^{3+2s}}dxdh=32\pi^3\sum_{k\in\mathbb{Z}^3}|u_k|^2\int_{\mathbb{R}^3}\frac{\sin^2(k.h/2)}{|h|^{3+2s}}dh=\\
\left|h:=\frac{z}{|k|}\right|=32\pi^3\sum_{k\in\mathbb{Z}^3}|u_k|^2|k|^{2s}\int_{\mathbb{R}^3}\frac{\sin^2\left(\frac{k}{2|k|}.z\right)}{|z|^{3+2s}}dz.
\end{multline}
Due to the rotation invariance of $|z|$, we have
\begin{equation}
4\int_{\mathbb{R}^3}\frac{\sin^2\left(\frac{k}{2|k|}.z\right)}{|z|^{3+2s}}dz=4\int_{\mathbb{R}^3}\frac{\sin^2(y_1/2)}{|y|^{3+2s}}dy:=c^{-1}.
\end{equation}
Since $s\in(0,1)$ the last integral is finite and, consequently,
$$
\int_{\mathbb{R}^3}\int_{\mathbb{T}^3}\frac{(u(x+h)-u(x))^2}{|h|^{3+2s}}dxdh=c^{-1}(2\pi)^3\sum_{k\in\Bbb Z^3}|u_k|^2|k|^{2s}=c^{-1}\|(-\Dx)^{s/2}u\|^2_{L^2}.
$$
Thus, Lemma \ref{Lem1.nice} is proved.
\end{proof}
Using \eqref{1.eqnorms} together with the obvious identity
$$
a(u,v)=\frac14(a(u+v,u+v)-a(u-v,u-v))
$$
which holds for any bilinear form $a(u,v)$, we conclude that, for any $u,v\in H^s(\Bbb T^3)$,
\begin{multline}\label{1.bigood}
(v,(-\Dx)^{s}v)=((-\Dx)^{s/2}u,(-\Dx)^{s/2}v)=\\=c\int_{\mathbb{R}^3}\int_{\mathbb{T}^3}\frac{(u(x+h)-u(x))(v(x+h)-v(x))}{|h|^{3+2s}}dxdh.
\end{multline}
In particular, taking $v=f(u)$ in the last formula and using that, due to the second assumption of \eqref{1.f} and the integral mean value theorem,
\begin{multline*}
(f(u(x+h))-f(u(x)))(u(x+h)-u(x))=\\=\int_{0}^1f'(\kappa u(x+h)+(1-\kappa)u(x))\,d\kappa|u(x+h)-u(x)|^2\ge-K|u(x+h)-u(x)|^2,
\end{multline*}
we see that
\begin{equation}\label{1.ffrac}
(f(u),(-\Dx)^s u)\ge-K\|(-\Dx)^{s/2}u\|_{L^2(\Bbb T^3)}^2\ge-C\|u\|^2_{H^s(\Bbb T^3)}
\end{equation}
hold at least for sufficiently smooth functions $u$ for which the integrals in the left and right-hand sides of \eqref{1.ffrac} have sense and, therefore, the key estimate \eqref{0.good} is verified. To be able to apply this estimate to less regular functions $u$ for which the existence of that integrals is not known a priori, we introduce, for every $\eb>0$, the cut-off kernels
\begin{equation}\label{1.phi}
\theta_\eb(z):=\begin{cases} |z|,\  \ |z|\ge\eb,\\ \eb,\ \ |z|\le\eb\end{cases}
\end{equation}
and the associated bilinear forms
\begin{equation}\label{1.es}
[u,v]_{s,\eb}:=\int_{\R^3}\int_{\Bbb T^3}\frac{(u(x+h)-u(x))(v(x+h)-v(x))}{\theta_\eb(h)^{3+2s}}\,dx\,dh.
\end{equation}
Then, on the one hand, obviously,
\begin{equation}\label{1.ez}
[u,v]_{s,\eb}^2\le [u,u]_{s,\eb}[v,v]_{s,\eb},\ \ [u,u]_{s,\eb}\le c^{-1}\|(-\Dx)^{s/2}u\|^2_{L^2(\Bbb T^3)}\le C\|u\|^2_{H^s(\Bbb T^3)},
\end{equation}
where $C$ is independent of $\eb>0$. Moreover, analogously to \eqref{1.ffrac}, we have
\begin{equation}\label{1.ffrace}
[f(u),u]_{s,\eb}\ge -K[u,u]_{s,\eb}.
\end{equation}
On the other hand, since the kernel $\frac{1}{\theta_\eb(h)}$ is no more singular, the integrals in \eqref{1.ffrace} have sense, for instance if $u\in L^p(\Bbb T^3)$ is such that $f(u)\in L^q(\Bbb T^3)$ with $\frac1p+\frac1q=1$. In the next section, we will use this estimate in the situation when $p=6$.
\par
We conclude the section by one more obvious lemma which allows us to estimate the $H^s$-norms using the smoothed norms \eqref{1.es}.
\begin{lemma}\label{Lem1.Fatou} Let the function $u\in L^2(\Bbb T^3)$ be such that
$$
[u,u]_{s,\eb}\le C
$$
for all $\eb>0$, where $C$ is independent of $\eb$. Then, $u\in H^s(\Bbb T^3)$ and
\begin{equation}\label{1.hs}
\|(-\Dx)^{s/2}u\|^2_{L^2}\le c\liminf_{\eb\to0}[u,u]_{s,\eb}.
\end{equation}
\end{lemma}
Indeed, the assertion of the lemma is an immediate corollary of the fact that $\eb\to\phi_\eb(h)$ is monotone increasing and the Fatou lemma.

\section{Energy solutions}\label{s2}

The aim of this section is to introduce and study the weak energy solutions for problem \eqref{OriginalSys}. In particular, as will be shown below, any such solution possesses an extra space-time regularity at least in the periodic case (the case of Dirichlet boundary conditions will be considered later in Section \ref{s6}). This extra regularity  will be essentially used later in order to verify existence, uniqueness and dissipativity of energy solutions.
\par
We start with recalling the basic energy equality for problem \eqref{OriginalSys}. Indeed, multiplying formally \eqref{OriginalSys} by $\partial_tu$ and integrating in $x\in\Omega$, after integration by parts (using the Dirichlet or periodic boundary conditions), we get
\begin{equation}\label{2.eneq}
\frac d{dt} E(u(t),\partial_t u(t))+\alpha\|\Dt u(t)\|^2_{L^2(\Omega)}+\gamma\|(-\Dx)^{1/4}\Dt u(t)\|^2_{L^2(\Omega)}=0,
\end{equation}
where
\begin{equation}\label{2.energy}
E(u,v)=\frac{1}{2}\|v\|^2_{L^2(\Omega)}++\frac{1}{2}\|\nabla u\|^2_{L^2(\Omega)}+(F(u),1)-(g,u)
\end{equation}
and $F(u):=\int_0^uf(s)\,ds$. Integrating \eqref{2.eneq} in time, we formally have
\begin{multline}\label{2.enint}
E(u(t),\Dt u(t))+\\+\int_0^t\alpha\|\Dt u(s)\|^2_{L^2(\Omega)}+\gamma\|(-\Dx)^{1/4}\Dt u(s)\|^2_{L^2(\Omega)}\,ds=E(u(0),\Dt u(0)).
\end{multline}
As usual, weak energy solutions are expected to have minimal regularity which guarantees that all terms in \eqref{2.enint} are well defined.
Note that, due to the growth restriction of \eqref{1.f},
$$
|F(u)|\le C(1+|u|^6)
$$
which together with the embedding $H^1\subset L^6$ show that the energy functional \eqref{2.energy} is naturally well-defined and continuous on the energy space $(u,v)\in\Cal E$. This justifies the following definition.

\begin{Def}\label{Def2.ensol} A function $\xi_u:=(u,\Dt u)$ is a weak energy solution of problem \eqref{OriginalSys} if
\begin{equation}\label{2.space}
\xi_u\in L^\infty(0,T;\Cal E),\ \ \Dt u\in L^2(0,T;H^{1/2}_\Delta),
\end{equation}
$\xi_u(0)=(u_0,u_1)$ and equation \eqref{OriginalSys} is satisfied in the sense of distributions. The latter means that,
for any $T>0$ and any $\phi\in C^\infty_0((0,T)\times\Omega)$,
\begin{multline}
\label{1.des}
-\int_0^T(\Dt u,\Dt\phi)dt+\int_0^T(\nabla u,\nabla \phi)dt+\alpha\int_0^T(u_t,\phi)dt\\
\gamma\int_0^T(\Dt u,(-\Dx)^{1/2}\phi)\,dt+
\int_0^T(f(u),\phi)dt=\int_0^T(g,\phi)\,dt.
\end{multline}
\end{Def}
Since $\xi_u(t)$ is a priori only in $L^\infty(0,T;\Cal E)$, the equality $\xi_u(0)=(u_0,u_1)$ requires some explanations. First, since $u,u_t\in L^\infty(0,T;L^2)$, we conclude that
$$
u\in C([0,T],L^2(\Omega))
$$
and the initial value for $u(t)$ has a sense. To verify the continuity in time of $\Dt u(t)$, we need the information on $\Dt^2 u(t)$. To this end,
we note that, using the growth restriction \eqref{1.f} on $f$ and Sobolev embedding $H^1\subset L^6$ one can easily verify that
$$
f(u)\in L^\infty(0,T;H^{-1}(\Omega))
$$
and then from \eqref{1.des}, we conclude that

\begin{equation}
\label{1.uttinf}
\Dt^2 u\in L^\infty(0,T;H^{-1}(\Omega)).
\end{equation}
Indeed, using definition of distributional derivative together with \eqref{1.des}, we have
\begin{multline}
\label{1.uttdis}
<\Dt^2 u, \phi>=-\int_0^T(\nabla u,\nabla \phi)dt-\\
\int_0^T(\Dt u,\alpha \phi+\gamma(-\Dx)^{1/2}\phi)dt-\int_0^T(f(u),\phi)dt+
\int_0^T(g,\phi)dt
\end{multline}
for any $\phi\in C^\infty_0((0,T)\times\Omega)$. Thus, we see that the distribution $\Dt^2 u$ can be extended by continuity to the linear functional on $L^1(0,T;H^1_0(\Omega))$ and \eqref{1.uttinf} now follows from the evident fact that $[L^1(0,T;H^1_0(\Omega))]^*=L^\infty(0,T;H^{-1}(\Omega))$.
Therefore,
$$
\Dt u\in C([0,T],H^{-1}(\Omega))
$$
and the initial data for $\Dt u$ is also well-defined.

\begin{remark}\label{Rem2.EnEq} Since
 $$
 \xi_u\in L^\infty(0,T;\Cal E)\cap C([0,T],\Cal E_{-1}),
 $$
 where $\Cal E_{-1}:=L^2(\Omega)\times H^{-1}(\Omega)$, we conclude that the value $\xi_u(t)$ is well defined for any $t\in[0,T]$ and the function $t\to \xi_u(t)$ is weakly continuous as a function with values in $\Cal E$:
\begin{equation}\label{2.weak}
\xi_u\in C_w([0,T],\Cal E),
\end{equation}
see e.g., \cite{Lions69}. Therefore, all terms in the energy equality \eqref{2.enint} indeed have sense for any $t\ge0$ and any weak energy solution $u$. Nevertheless, the regularity \eqref{2.space} is a priori not enough for establishing the validity of the energy equality since we do not have enough regularity to take $\phi=\Dt u$ in \eqref{1.des} (the most difficult is the term $(f(u),\Dt u)$, see below), so to the best of our knowledge neither the validity of energy equality nor the strong continuity of the function $t\to\xi_u(t)$ in $\Cal E$ was known before for the energy solutions of \eqref{OriginalSys} with quintic nonlinearity. Below we will establish these facts based on \eqref{1.ffrac} and the extra energy type functional.
\end{remark}
The next theorem gives the existence of energy solutions and their dissipativity.

\begin{theorem}\label{Th2.existence} Let the nonlinearity $f$ satisfies \eqref{1.f}, $\alpha,\gamma>0$, $g\in L^2(\Omega)$ and $\xi_u(0)=(u_0,u_1)\in\Cal E$. Then problem \eqref{OriginalSys} (endowed by Dirichlet or periodic boundary conditions) possesses at least one weak energy solution $\xi_u(t)$ in the sense of Definition \ref{Def2.ensol} which satisfies the following estimate:
\begin{equation}\label{2.dis}
\|\xi_u(t)\|_{\Cal E}^2+\int_t^{t+1}\|\Dt u(s)\|^2_{H^{1/2}_\Delta}\,ds\le Q(\|\xi_u(0)\|)e^{-\beta t}+Q(\|g\|_{L^2(\Omega)}),
\end{equation}
where $\beta>0$ and $Q$ is a monotone increasing function which is independent of $t$ and $u$.
\end{theorem}
The proof of this theorem is completely standard for the theory of damped wave equations. Indeed, to derive \eqref{2.dis} formally it is sufficient to multiply equation \eqref{OriginalSys} by $\Dt u+\eb u$ for some positive $\eb$, integrate by parts and apply the Gronwall inequality. The existence of a solution as well as the justification of estimate \eqref{2.dis} can be done using, say, the Galerkin approximations, see \cite{BV,Lions69} for the details.

We are now ready to state the main result of this section on the extra regularity of energy solutions. For simplicity, we restrict ourselves to consider first only the case of periodic boundary conditions (the analogous result for Dirichlet boundary conditions will be obtained in Section \ref{s6}).

\begin{theorem}\label{Th2.reg} Let $\Omega:=\Bbb T^3$ (with periodic boundary conditions). Then, any weak energy solution $u$ of \eqref{OriginalSys} in the sense of Definition \ref{Def2.ensol} belongs to the space $L^2(0,T;H^{3/2}(\Omega))$ and the following estimate holds:
\begin{multline}\label{2.stest}
\|u\|_{L^2([\tau,\tau+1],H^{3/2}(\Omega))}\le\\\le Q(\|\xi_u\|_{L^\infty([\tau,\tau+1],\Cal E)}+\|\Dt u\|_{L^2([\tau,\tau+1],H^{1/2}_\Delta)}+\|g\|_{L^2(\Omega)}),
\end{multline}
where the monotone function $Q$ is independent of $\tau\in\R_+$ and $u$.
\end{theorem}
\begin{proof} We give first the formal derivation of estimate \eqref{2.stest}. To this end, we multiply equation \eqref{OriginalSys} by $(-\Dx)^{1/2}u$ and integrate over $x\in\Omega$. Then, integrating by parts and using \eqref{1.ffrac}, we get
\begin{multline}
\frac d{dt}\((\Dt u,(-\Dx)^{1/2}u)+\frac\gamma2\|(-\Dx)^{1/2}u\|^2_{L^2(\Omega)}+\frac\alpha2\|(-\Dx)^{1/4}u\|^2_{L^2(\Omega)}\)+\\+\|(-\Dx)^{3/4}u\|^2_{L^2(\Omega)}\le C(\|g\|^2_{L^2(\Omega)}+\|u\|^2_{H^1(\Omega)})+\|(-\Dx)^{1/4}\Dt u\|^2_{L^2(\Omega)}).
\end{multline}
Integrating this inequality in time $t\in[\tau,\tau+1]$ and using the Cauchy-Schwartz inequality again, we end up with
\begin{multline}\label{2.long}
\int_\tau^{\tau+1}\|(-\Dx)^{3/4}u(s)\|^2_{L^2(\Omega)}\,ds\le \int_\tau^{\tau+1}\|(-\Dx)^{1/4}\Dt u(s)\|^2_{L^2(\Omega)}\,ds+\\+ C\(\|g\|^2_{L^2(\Omega)}+\|u(\tau)\|^2_{H^1(\Omega)}+\|u(\tau+1)\|^2_{H^1(\Omega)}+\|\Dt u(\tau)\|^2_{L^2(\Omega)}+\right.\\\left.+\|\Dt u(\tau+1)\|^2_{L^2(\Omega)}+\int_{\tau}^{\tau+1}\|u(s)\|^2_{H^1(\Omega)}\,ds\)
\end{multline}
which gives the desired estimate \eqref{2.stest}. Thus, it only remains to justify the above estimates. To this end, we will take the inner product \eqref{1.es} with $s=1/2$ and $\eb>0$ of equation \eqref{OriginalSys} with $u$. This, is equivalent to taking the test function
$$
\phi_{h,\eb}(x):=-\frac{u(x+h)-2u(x)+u(x-h)}{\theta_{\eb}(h)^4}
$$
in \eqref{1.uttdis} with further integration with respect to $h\in\R^3$. Since the function $\phi_{h,\eb}\in L^1(0,T;H^1_0(\Omega))$, then this integration is justified. Using now \eqref{1.ez} and \eqref{1.ffrace} together with the standard formulas
$$
[\Dx u,u]_{1/2,\eb}=-[\Nx u,\Nx u]_{1/2,\eb},\ \ [\Dt^2 u,u]_{1/2,\eb}=\Dt[\Dt u,u]_{1/2,\eb}-[\Dt u,\Dt u]_{1/2,\eb},
$$
we end up with the following analogue of \eqref{2.long}:
\begin{multline}\label{2.llong}
\int_\tau^{\tau+1}[\Nx u(s),\Nx u(s)]_{1/2,\eb}\,ds\le \\\le[\Dt u(s),u(s)]_{1/2,\eb}\big|_{s=\tau}^{s=\tau+1}+\\+C(\|g\|^2_{L^2}+\|u(\tau)\|^2_{H^1(\Omega)})+C\int_\tau^{\tau+1}\|\Dt u(s)\|^2_{H^{1/2}(\Omega)}+\|u(s)\|^2_{H^1(\Omega)}\,ds.
\end{multline}
We estimate the middle term in \eqref{2.llong} as follows:
\begin{multline}
[\Dt u,u]_{1/2,\eb}=[(-\Dx+1)^{-1/4}\Dt u,(-\Dx+1)^{1/4}u]_{1/2,\eb}\le\\\le [(-\Dx+1)^{-1/4}\Dt u,(-\Dx+1)^{-1/4}\Dt u]_{1/2,\eb}+
[(-\Dx+1)^{1/4}u,(-\Dx+1)^{1/4}u]_{1/2,\eb}\le\\\le C(\|(-\Dx)^{1/4}(-\Dx+1)^{-1/4}\Dt u\|^2_{L^2}+\|(-\Dx)^{1/4}(-\Dx+1)^{1/4}u\|^2_{L^2})\le C\|\xi_u\|^2_{\Cal E}.
\end{multline}
Inserting this estimate in \eqref{2.llong}, passing to the limit $\eb\to0$ and using Lemma \ref{Lem1.Fatou}, we justify estimate \eqref{2.long} and finish the proof of the theorem
\end{proof}
\begin{remark}\label{rem2.supercrit}
We emphasize that the above proof works only in the case of periodic boundary conditions although, as we will see later, estimate \eqref{2.stest} remains true also for the case of Dirichlet boundary conditions. On the other hand, the above given proof essentially uses only the fact that $f'(u)\ge-K$ and the growth restrictions on $f$ is nowhere essentially used, so the above result remains true in the supercritical case of faster than quintic growth rate as well.
\end{remark}
\begin{cor} Let the energy solution of problem \eqref{OriginalSys} satisfy estimate \eqref{2.stest}. Then, for every $s\in[0,1)$, $u\in L^{2/s}(0,T;L^{6/(1-s)}(\Omega))$ and the following estimate holds:
\begin{equation}\label{2.s-t-dis}
\|u\|_{L^{2/s}([t,t+1],L^{6/(1-s)}(\Omega))}\le Q_s(\|\xi_u(0)\|_{\Cal E})e^{-\beta t}+Q(\|g\|_{L^2(\Omega)}),
\end{equation}
where the monotone function $Q_s$ depends on $s$, but is independent of $t$ and $u$.
\end{cor}
Indeed, \eqref{2.s-t-dis} follows from \eqref{2.stest}, \eqref{2.dis} and the interpolation inequality
$$
\|u\|_{L^{2/s}([t,t+1],L^{6/(1-s)}(\Omega))}\le C_s\|u\|^{1-s}_{L^\infty([t,t+1],H^1(\Omega))}\|u\|^s_{L^2([t,t+1],H^{3/2}(\Omega))}.
$$
The next result shows that the extra regularity \eqref{2.stest} is enough to verify the energy equality.

\begin{cor}\label{Cor2.eneq} Let the energy solution $u(t)$ of problem \eqref{OriginalSys} (with periodic or Dirichlet boundary conditions) satisfy estimate \eqref{2.stest}. Then the function $t\to E(u(t),\Dt u(t))$ is absolutely continuous and the energy identity \eqref{2.eneq} holds for almost all $t$. In particular, $\xi_u\in C([0,T],\Cal E)$.
\end{cor}
\begin{proof} Indeed, using the \eqref{2.s-t-dis} with $s=\frac15$, the embedding $H^{1/2}_\Delta\subset L^3(\Omega)$ and growth restriction \eqref{1.f} on the nonlinearity $f$, we see that
$$
\|f(u)\|_{L^2(t,t+1;H^{-1/2}_\Delta)}\le\! C\|f(u)\|_{L^2(t,t+1;L^{3/2}(\Omega))}\le\! C(1+\|u\|_{L^{10}(t,t+1;L^{15/2}(\Omega))}^5)\le C_u
$$
and, therefore, $f(u)\in L^2(0,T;H^{-1/2}_\Delta)$. Moreover, due to estimate \eqref{2.stest}, $\Dx u\in L^2(0,T;H^{-1/2}_\Delta)$ and, due to Definition \ref{Def2.ensol}, $(-\Dx)^{1/2}\Dt u\in L^2(0,T;H^{-1/2}_\Delta)$. Thus, from equation \eqref{OriginalSys}, we derive that $\Dt^2 u\in L^2(0,T; H^{-1/2}_\Delta)$ as well.
\par
Since all terms of equation \eqref{OriginalSys} belong to the space $L^2(0,T;H^{-1/2}_\Delta)$ and
$$
\Dt u\in L^2(0,T;H^{1/2}_\Delta)=[L^2(0,T;H^{-1/2}_\Delta)]^*,
$$
 the multiplication of the equation on $\Dt u$ (= taking the test function $\phi=\Dt u$ in \eqref{1.des}) is allowed. In addition, this regularity implies that the functions $t\to \|\Dt u(t)\|_{L^2}$ and $t\to\|\Nx u(t)\|^2_{L^2}$ are absolutely continuous and
$$
(\Dt^2 u(t),\Dt u(t))=\frac12\frac d{dt}\|\Dt u(t)\|^2_{L^2},\ \ (\Nx u(t),\Nx\Dt u(t)) =\frac12\frac d{dt}\|\Nx u(t)\|^2_{L^2},
$$
for almost all $t$, see e.g., \cite{TemamDS}. Finally, approximating the function $u$ by smooth functions and passing to the limit, one verifies that the function $t\to (F(u(t)),1)$ is absolutely continuous and
$$
\frac d{dt}(F(u(t)),1)=(f(u),\Dt u)
$$
for almost all $t$. Thus, the energy equality is proved.
 The continuity of $\xi_u(t)$ as a $\Cal E$-valued function follows in a standard way from the energy equality and the corollary is proved.

\end{proof}

\section{Uniqueness and smoothing property}\label{s4}
In that section we show that the extra regularity \eqref{2.stest} is sufficient to verify the well-posedness and smoothness of energy solutions. Everywhere in that and next sections we assume that any energy solution satisfies \eqref{2.stest}. As we have already seen that is the case when the periodic boundary conditions are posed (as will be shown below, that is also true in the case of Dirichlet boundary conditions under some additional technical assumptions). We start with the uniqueness result.

\begin{theorem}
\label{Th4.uni}
Let the assumptions of Theorem \ref{Th2.existence} hold and let, in addition, all energy solutions satisfy \eqref{2.stest}. Then the energy solution of problem \eqref{OriginalSys} is unique and for the difference $v(t)$ of two energy solutions
$u_1(t)$ and $u_2(t)$ the following estimate is valid:
\begin{equation}
\label{4.lip}
\|v(t)\|^2_{H^1(\Omega)}+\|\partial_tv(t)\|^2_{L^2(\Omega)}\leq e^{\hat{K}t}(\|v(0)\|^2_{H^1(\Omega)}+\|\partial_tv(0)\|^2_{L^2(\Omega)}),\ \forall t\in[0,T],
\end{equation}
where $\hat{K}$ is a positive constant depending on $\|\xi_{u_1}(0)\|_\Cal E$, $\|\xi_{u_2}(0)\|_\Cal E$ and $\|g\|_{L^2(\Omega)}$.
\end{theorem}
\begin{proof}
The function $v(t)$ as difference of two solutions $u_1(t)$ and $u_2(t)$ solves
\begin{equation}
\label{differEq}
\begin{cases}
\Dt^2v+\alpha\Dt v+\gamma(-\Delta)^{1/2}\Dt v-\Dx v+f(u_1)-f(u_2)=0,\\
\xi_v(0)=\xi_{u_1}(0)-\xi_{u_2}(0).
\end{cases}
\end{equation}
Since both $u_1$ and $u_2$ belong to $L^2(0,T;H^{3/2}(\Omega))$,
arguing as in Corollary \ref{Cor2.eneq}, we can justify the  multiplication of \eqref{differEq} by $\Dt v$ and obtain the following identity
\begin{multline}
\label{DifferEq*vt}
\frac12\frac{d}{dt}\|\xi_v(t)\|^2_{\Cal E}+\gamma\|(-\Dx)^\frac{1}{4}v_t\|^2_{L^2}+\alpha\|\Dt v(t)\|^2_{L^2}= -(f(u_1)-f(u_2),\Dt v).
\end{multline}
We estimate the right-hand side of  \eqref{DifferEq*vt} using the integral mean value theorem, growth restriction \eqref{1.f} on $f'$, and the H\"older inequality with exponents $3$, $6$ and~$2$:
\begin{multline}
\label{nonlinTermEst}
|(f(u_1)-f(u_2),v_t)|\le\left(\int_0^1|f'(\lambda u_2+(1-\lambda)u_1)|d\lambda,|v||v_t|\right)\leq\\
C\left((1+|u_1|^4+|u_2|^4),|v||v_t|\right)\le
C(1+\|u_1\|_{L^{12}(\Omega)}^4+\|u_2\|^4_{L^{12}(\Omega)})\|v\|_{H^1}\|\Dt v\|_{L^2}\le\\\le C(1+\|u_1\|_{L^{12}(\Omega)}^4+\|u_2\|^4_{L^{12}(\Omega)})\|\xi_v\|^2_{\Cal E}.
\end{multline}
Combining \eqref{DifferEq*vt} and \eqref{nonlinTermEst}, we have
\begin{multline}\label{4.main}
\frac{d}{dt}\|\xi_v(t)\|^2_{\Cal E}+2\gamma\|(-\Dx)^{1/4}\Dt v(t)\|^2_{L^2(\Omega)}+2\alpha\|\Dt v(t)\|^2_{L^2(\Omega)}\leq\\
C(1+\|u_1(t)\|_{L^{12}(\Omega)}^4+\|u_2(t)\|_{L^{12}(\Omega)}^4)\|\xi_v(t)\|^2_{\Cal E}.
\end{multline}
Using now \eqref{2.s-t-dis} with $s=1/2$, we get
\begin{equation}\label{4.412}
\|u_i\|_{L^4(0,T;L^{12}(\Omega))}^4\le (Q(\|\xi_{u_i}(0)\|_{\Cal E})+Q(\|g\|_{L^2(\Omega)}))(T+1),\ \ i=1,2,
\end{equation}
and the Gronwall inequality applied to \eqref{4.main} gives the desired estimate \eqref{4.lip} and finishes the proof of the theorem.
\end{proof}

The next proposition gives us additional smoothness of solutions assuming that the initial data is more regular. We start with the estimate which is divergent as time tends to infinity, the analogous dissipative estimate will be obtained later.

\begin{prop}\label{Prop4.e1smooth} Let assumptions of Theorem \ref{Th4.uni} hold and the initial data
\begin{equation}
\label{4.e1ini}
\xi_u(0)\in\Cal E_1:=[H^2(\Omega)\cap H^1_0(\Omega)]\times H^1_0(\Omega).
\end{equation}
Then $\xi_{u}(t)$ belongs to $\Cal E_1$ for any positive $t$ and the following estimate holds:
\begin{equation}\label{4.e1div}
\|\xi_{\Dt u}(t)\|^2_{\Cal E}+\|\xi_u(t)\|^2_{\Cal E^1}\le
Q(\|\xi_u(0)\|_{\Cal E_1}+\|g\|_{L^2(\Omega)})e^{Q(\|\xi_u(0)\|_{\Cal E}+\|g\|_{L^2(\Omega)})(t+1)},
\end{equation}
for some monotone increasing function $Q$ which is independent of $t\ge0$ and $u$.
\end{prop}
\begin{proof} We give below only the formal derivation of estimate \eqref{4.e1div} which can be justified in a standard way using e.g., the Galerkin approximations. To this end, we differentiate \eqref{OriginalSys} in time and denote $v(t):=\Dt u(t)$. Then this function solves
\begin{equation}\label{4.diffeq}
\Dt^2 v+\alpha\Dt v+\gamma(-\Dx)^{1/2}\Dt v-\Dx v+f'(u)v=0,\ \ \xi_v(0):=(\Dt u(0),\Dt^2 u(0)),
\end{equation}
where $\Dt^2 u(0)$ is defined via
\begin{equation}\label{4.d2t0}
\Dt^2 u(0):=\Dx u(0)-\alpha\Dt u(0)-\gamma(-\Dx)^{1/2}\Dt u(0)-f(u(0))+g.
\end{equation}
Using the embedding $H^2\subset C$, it is not difficult to see that
$$
\|\Dt^2 u(0)\|_{L^2(\Omega)}\le Q(\|\xi_u(0)\|_{\Cal E_1}+\|g\|_{L^2(\Omega)})
$$
and, therefore,
$$
\|\xi_v(0)\|_{\Cal E}\le Q(\|\xi_u(0)\|_{\Cal E_1}+\|g\|_{L^2(\Omega)})
$$
for some monotone increasing function $Q$. Analogously, for every $t\ge0$, we have
\begin{equation}\label{4.vu}
\|\xi_v(t)\|_{\Cal E}\le Q(\|\xi_u(t)\|_{\Cal E_1}+\|g\|_{L^2(\Omega)}),
\end{equation}
where $Q$ is independent of $t$. Vice versa, multiplying equation \eqref{OriginalSys} by $-\Dx u(t)$ (for every fixed $t$), integrating by parts and using that $f'\ge-K$, we end up with
$$
\|\Dx u(t)\|_{L^2(\Omega)}^2\le C\(\|\xi_v(t)\|^2_{\Cal E}+\|u(t)\|^2_{H^1(\Omega)}+\|g\|^2_{L^2(\Omega)}\)
$$
and, therefore,
\begin{equation}\label{4.uv}
\|\xi_u(t)\|_{\Cal E_1}^2\le C\(\|\xi_v(t)\|^2_{\Cal E}+\|\xi_u(t)\|^2_{\Cal E}+\|g\|^2_{L^2(\Omega)}\).
\end{equation}
Thus, to verify \eqref{4.e1div} it is sufficient to estimate the quantity $\|\xi_v(t)\|_{\Cal E}$ only. To this end, we multiply equation \eqref{4.diffeq} by $\Dt v$ and integrate over $\Omega$. Then, arguing as in the derivation of \eqref{4.main}, we end up with
\begin{equation}\label{4.mmain}
\frac d{dt}\|\xi_v(t)\|^2_{\Cal E}+\delta\|\Dt v(t)\|^2_{H^{1/2}_\Delta}\le C(1+\|u(t)\|^4_{L^{12}(\Omega)})\|\xi_v(t)\|^2_{\Cal E},
\end{equation}
for some $\delta>0$. Applying the Gronwall inequality to \eqref{4.mmain} and using the estimate \eqref{4.412} for the $L^4(L^{12})$-norm of $u$ together with \eqref{4.uv} and \eqref{4.vu}, we end up with the desired estimate \eqref{4.e1div} and finish the proof of the proposition.
\end{proof}

The next proposition gives the parabolic smoothing property for the energy solutions.

\begin{prop}
\label{Prop4.smooth}
Let assumptions of Theorem \ref{Th4.uni} hold. Then any energy solution $u(t)$ of problem \eqref{OriginalSys}  possesses the following smoothing property for $t\in(0,1]$:
\begin{equation}\label{4.smo}
t^2\(\|\xi_{\Dt u}(t)\|^2_{\Cal E}+\|\xi_u(t)\|^2_{\Cal E_1}\)\le Q(\|\xi_u(0)\|_{\Cal E})+Q(\|g\|_{L^2(\Omega)}),
\end{equation}
for some monotone increasing function $Q$ which is independent of $t$ and $u$.
\end{prop}
\begin{proof} Again, we give only the formal derivation of estimate \eqref{4.smo} which can be justified, say, by the Galerkin approximations.
Multiplying inequality \eqref{4.mmain} for $v:=\Dt u$ by $t^2$ and using that
\begin{multline*}
2t\|\Nx v(t)\|^2_{L^2(\Omega)}=2\frac d{dt}\(t(\Nx u(t),\Nx v(t))\)-2(\Nx u(t),\Nx v(t))-\\-2t((-\Dx)^{3/4}u(t),(-\Dx)^{1/4}\Dt v(t))
\end{multline*}
and that
\begin{multline*}
2t\|\Dt v(t)\|^2_{L^2(\Omega)}\le Ct\|\Dt v(t)\|_{H^{-1/2}_\Delta}\|\Dt v(t)\|_{H^{1/2}_\Delta}\le\\\le \frac{\delta t^2}{2}\|\Dt v(t)\|^2_{H^{1/2}_\Delta}+C\|\Dt^2 u(t)\|_{H^{-1/2}_\Delta}^2,
\end{multline*}
after the elementary estimates, we get
\begin{multline}\label{4.mmmain}
\frac d{dt}\(t^2\|\xi_v(t)\|^2_{\Cal E}-2t(\Nx u(t),\Nx v(t))\)-\\-C(1+\|u(t)\|^4_{L^{12}(\Omega)})(t^2\|\xi_v(t)\|_{\Cal E}^2)\le C\(\|\Dt^2 u(t)\|^2_{H^{-1/2}_\Delta}+\|u(t)\|^2_{H^{3/2}(\Omega)}\).
\end{multline}
Integrating this inequality in time and using that the norms $\|\Dt u\|_{L^2(0,1;H^{1/2}_{\Delta})}$ and $\|\Dt^2 u\|_{L^2(0,1;H^{-1/2}_\Delta)}$ are under the control (see the proof of Corollary \ref{Cor2.eneq} concerning the second norm) as well as the inequality
$$
|2t(\Nx u(t),\Nx v(t))|\le \frac12 t^2\|\xi_v(t)\|^2_{\Cal E}+2\|\xi_u(t)\|^2_{\Cal E},
$$
 we end up with
$$
t^2\|\xi_v(t)\|^2_{\Cal E}\le 2C\int_0^t(1+\|u(s)\|^4_{L^{12}(\Omega)})s^2\|\xi_v(s)\|^2_{\Cal E}\,ds+ Q(\|\xi_u(0)\|_{\Cal E}+\|g\|_{L^2(\Omega)}).
$$
Using finally that the $L^4(0,1;L^{12}(\Omega))$-norm of $u$ is also under the control, see \eqref{4.412}, and applying the Gronwall inequality to the last estimate, we end up with the desired estimate \eqref{4.smo} for the $\Cal E$-norm of $\xi_v(t)$. The analogous estimate for the $\Cal E_1$-norm of $\xi_u(t)$ is now an immediate corollary of \eqref{4.uv} and the proposition is proved.
\end{proof}
We are now ready to establish the dissipativity of \eqref{OriginalSys} in $\Cal E_1$.

\begin{cor}\label{Cor4.e1dis} Let the assumptions of Theorem \ref{Th4.uni} hold and let, in addition, $\xi_u(0)\in\Cal E_1$. Then, the following estimate holds:
\begin{equation}\label{4.e1dis}
\|\xi_u(t)\|_{\Cal E_1}+\|\xi_{\Dt u}(t)\|_{\Cal E}\le Q(\|\xi_u(0)\|_{\Cal E_1})e^{-\beta t}+Q(\|g\|_{L^2(\Omega)}),
\end{equation}
where  $\beta>0$ and the monotone function $Q$ are independent of $t$ and $\xi_u(0)$.
\end{cor}
\begin{proof} Indeed, according to Proposition \ref{Prop4.smooth},
\begin{equation}\label{4.tsmo}
\|\xi_u(t+1)\|_{\Cal E_1}\le Q(\|\xi_u(t)\|_{\Cal E})+Q(\|g\|_{L^2(\Omega)})
\end{equation}
for some monotone function $Q$ independent of $t$ and $u$. Combining this estimate with the dissipative estimate \eqref{2.dis}, we get the desired
estimate \eqref{4.e1dis} for $t\ge1$. To obtain estimate \eqref{4.e1dis} for $t\le1$, it is sufficient to use Proposition \ref{Prop4.e1smooth}. Thus, the corollary is proved.
\end{proof}

\section{The  attractors}\label{s5}
In that section, we prove the existence of global and exponential attractors for problem \eqref{OriginalSys} under the additional assumption that any energy solution possesses the additional regularity \eqref{2.stest}.
We start with summarizing the important properties of the energy solutions obtained above. First, according to Theorem \ref{Th4.uni}, in that case problem \eqref{OriginalSys} generates a semigroup $S(t)$ in the energy phase space $\Cal E$ via
\begin{equation}
S(t):\Cal E\to\Cal E,\ \ S(t)\xi:=\xi_u(t),
\end{equation}
where $\xi_u(t)$ is a unique energy solution of \eqref{OriginalSys} such that $\xi_u(0)=\xi\in\Cal E$. Second, this semigroup is dissipative due to estimate \eqref{2.dis} and is locally Lipschitz continuous due to estimate \eqref{4.lip}. Third, due to estimates \eqref{4.e1dis} and \eqref{4.tsmo}, the ball $\Cal B_{R}$ of a sufficiently large radius $R$ of $\Cal E_1$ will be a compact in $\Cal E$ (since $\Cal{E}$ is reflexive and $\Cal{B}_R$ is convex, $\Cal B_R$ is closed in $\Cal E$ as well) absorbing set for this semigroup, i.e., for every bounded set $B$ in $\Cal E$ there is time $T=T(B)$ such that
$$
S(t)B\subset {\Cal B}_R, \ \ \forall t\ge T.
$$
As usual, based on $\Cal B_{R}$ one can construct the {\it semi-invariant} compact absorbing set for $S(t)$ via
\begin{equation}\label{5.s-absorb}
\Cal B:=\cup_{t\ge0}S(t)\Cal B_{R},\ \ S(t)\Cal B\subset \Cal B.
\end{equation}
Indeed, since $\Cal B_R\subset\Cal B$, the set $\Cal B$ is also an absorbing set for $S(t)$. This set is bounded in $\Cal E_1$ due to estimate \eqref{4.e1dis} and its closedness in $\Cal E$ (and, therefore, in $\Cal E_1$ as well) is evident. Thus, $\Cal B$ is a compact set in $\Cal E$ and is a semi-invariant absorbing set for the semigroup $S(t)$.
\par
For the convenience of the reader, we now recall the definition of a global attractor, see \cite{BV,MZDafer2008,TemamDS} for more details.
\begin{Def}\label{Def5.attr} A set $\Cal A\subset\Cal E$ is a global attractor for the semigroup $S(t)$ in $\Cal E$ if:
\par
1) The set $\Cal A$ is compact in $\Cal E$.
\par
2) The set $\Cal A$ is strictly invariant: $S(t)\Cal A=\Cal A$, $t\ge0$.
\par
3) The set $\Cal A$ is an attracting set for the semigroup $S(t)$ in $\Cal E$, i.e., for any bounded set $B$ in $\Cal E$ and any neighborhood $\Cal O(\Cal A)$ of the attractor $\Cal A$ in $\Cal E$ there is time $T=T(B,\Cal E)$ such that
$$
S(t)B\subset \Cal O(\Cal A),\ \ t\ge T.
$$
\end{Def}
Remind that the third property in Definition \ref{Def5.attr} can be reformulated in the equivalent way using the so-called non-symmetric Hausdorff distance between sets, namely,
\begin{equation}\label{5.conv}
\lim_{t\to\infty}\dist_{\Cal E}(S(t)B,\Cal A)=0,
\end{equation}
for any bounded set $B\subset\Cal E$. Here and below the Hausdorff distance between sets $X$ and $Y$ in the space $\Cal E$ is defined as follows:
$$
\dist_{\Cal E}(X,Y):=\sup_{x\in X}\inf_{y\in Y}\|x-y\|_{\Cal E}.
$$
\begin{prop} Let the assumptions of Theorem \ref{Th4.uni} hold. Then the solution semigroup $S(t)$ in $\Cal E$ of problem \eqref{OriginalSys} possesses the global attractor $\Cal A$ which is bounded in $\Cal E_1$ and is generated by all  trajectories of $S(t)$ which are defined for all $t\in\R$ and bounded in $\Cal E$:
\begin{equation}\label{6.k}
\Cal A=\Cal K\big|_{t=0},
\end{equation}
where $\Cal K\in C_b(\R,\Cal E)$ is the set all bounded solutions of \eqref{OriginalSys} defined for all $t\in\R$.
\end{prop}
\begin{proof} According to the abstract attractor's existence theorem, see e.g., \cite{BV}, we need to check that a) The solution semigroup $S(t)$ possesses a compact absorbing set; b) the operators $S(t)$ are continuous in $\Cal E$ for every fixed $t$. Since both of these assertions are already verified above, the existence of the global attractor is also verified. Since the constructed absorbing set $\Cal B$ is bounded in $\Cal E_1$ and the attractor is a subset of $\Cal B$, it is also bounded in $\Cal E_1$. Finally, the representation \eqref{6.k} is also a standard corollary of the attractor's existence theorem. Thus, the proposition is proved.
\end{proof}
As the next step, we intend to verify the finite-dimensionality of the global attractor $\Cal A$. To this end, we recall the definition of the fractal (box-counting) dimension.

\begin{Def} Let $K$ be a compact set in a metric space $\Cal E$. By Hausdorff criterium, for every $\eb>0$, $K$ can be covered by finitely-many balls of radius $\eb$ in $\Cal E$. Let $N_\eb(K,\Cal E)$ be the minimal number of such balls which is enough to cover $\Cal E$. Then, the fractal dimension of $K$ is defined as follows:
\begin{equation}\label{5.frac}
\dim_f(K,\Cal E):=\limsup_{\eb\to0}\frac{\log N_\eb(K,\Cal E)}{\log\frac1\eb}.
\end{equation}
\end{Def}
Of course, $\dim_f(K,\Cal E)$ a priori can be infinite if $\Cal E$ is infinite-dimensional.
We also remind that, in the case when $K$ is a finite-dimensional Lipschitz manifold in $\Cal E$, the fractal dimension coincides with the usual dimension, but for  irregular sets (which is often the case where $K=\Cal A$ is a global attractor), the fractal dimension may be not integer, see e.g., \cite{robinson} for more details.
\par
We check the finite-dimensionality of the global attractor $\Cal A$ for problem \eqref{OriginalSys}  by constructing one more important object - the so-called exponential attractor which has been introduced in \cite{EFNT} in order to overcome the main drawback of a global attractor, namely, the absence of control for the rate of convergence in \eqref{5.conv}.

\begin{Def} A set $\Cal M$ is an exponential attractor for the semigroup $S(t)$ in $\Cal E$ if the following conditions are satisfied:
\par
1) The set $\Cal M$ is compact in $\Cal E$.
\par
2) The set is $\Cal M$ is semi-invariant: $S(t)\Cal M\subset\Cal M$.
\par
3) The set $\Cal M$ has finite fractal dimension in $\Cal E$.
\par
4) The set $\Cal M$ attracts exponentially the images of bounded sets, i.e., for every bounded set $B$ in $\Cal E$,
\begin{equation}\label{5.expattr}
\dist_{\Cal E}(S(t)B,\Cal M)\le Q(\|B\|_{\Cal E})e^{-\beta t},\ \ t\ge0,
\end{equation}
for some positive $\beta$ and monotone function $Q$ which are independent of $t$.
\end{Def}
Roughly speaking, the exponential rate of attraction \eqref{5.expattr} is achieved by adding to the global attractor a number of "metastable" trajectories (which approach it too slowly) and the non-trivial result of the exponential attractors theory is the possibility to do that without destroying the finite-dimensionality (in almost all cases where the finite-dimensionality of the global attractor is established, see \cite{MZDafer2008} for more details). The control of the rate of convergence \eqref{5.expattr}, in particular, makes an exponential attractor much more robust with respect to perturbations. For instance, in contrast to a global one, an exponential attractor is as a rule upper and lower semincntinuous (and even H\"older continuous) with respect to the perturbation parameter, see \cite{MZDafer2008}. As the price to pay, an exponential attractor is not {\it unique} (similar to center/inertial manifolds), although this drawback can be partially overcame using the proper selection
of one-valued branches of exponential attractors in dependence of the perturbation parameters, see \cite{EMZ05,MZDafer2008} and the references therein for more details.
\par
The next theorem, which establishes the existence of an exponential attractor for problem \eqref{OriginalSys}, can be considered as the main result of the section.
\begin{theorem}\label{Th5.expoattr}
Let assumptions of Theorem \ref{Th4.uni} hold.  Then the solution semigroup associated with problem \eqref{OriginalSys} possesses an exponential attractor $\Cal M$ which is bounded in the space $\Cal E_1$.
\end{theorem}
\begin{proof} As usual, it is enough to construct an exponential attractor for the semigroup $S(t)$ restricted to the semi-invariant absorbing set $\Cal B$ defined by \eqref{5.s-absorb} only. Also as usual, we start with constructing the exponential attractor $\Cal M_d$ for the map $S=S(1):\Cal B\to\Cal B$ which will be upgraded after that to the desired exponential attractor for the case of continuous time. To this end, we need the following lemma.

\begin{lemma}\label{Lem5.difsmo} Let the above assumptions hold. Then, for any $\xi_1,\xi_2\in\Cal B$, the following is true:
\begin{equation}
\|S(1)\xi_1-S(1)\xi_2\|_{\Cal E_{1/2}}\le L\|\xi_1-\xi_2\|_{\Cal E},
\end{equation}
where $\Cal E_{1/2}:=[H^{3/2}(\Omega)\cap H^1_0(\Omega)]\times H^{1/2}_\Delta$ and the constant $L$ is independent of $\xi_1,\xi_2\in\Cal B$.
\end{lemma}
\begin{proof}[Proof of the lemma] Indeed, let $\xi_{u_i}(t):=S(t)\xi_i$, $i=1,2$, be two trajectories starting from $\xi_1,\xi_2\in\Cal B$ and let $v(t)=u_1(t)-u_2(t)$. Then, this function satisfies equation \eqref{DifferEq*vt}. Moreover, using that $\Cal B$ is bounded in $\Cal E_1$ and the embedding $H^2\subset C$, analogously to \eqref{nonlinTermEst}, we derive that
\begin{equation}\label{5.dif12}
\|f(u_1)-f(u_2)\|_{L^2(\Omega)}\le C\|u_1-u_2\|_{L^2(\Omega)},
\end{equation}
where $C$ is independent of $\xi_i\in\Cal B$. Applying Cauchy-Schwartz inequality to the term on the right of \eqref{DifferEq*vt}, using\eqref{5.dif12}, we have
$$
\frac{d}{dt}\|\xi_v(t)\|_{\Cal E}^2+\|\Dt v(t)\|^2_{H^{1/2}_\Delta}\le C_1\|\xi_v(t)\|^2_{\Cal E}
$$
and, due to the Gronwall inequality, we end up with
\begin{equation}\label{5.lip}
\|\xi_v(t)\|^2_{\Cal E}+\int_0^t\|\Dt v(t)\|^2_{H^{1/2}_\Delta}\,dt\le C_2\|\xi_v(0)\|^2_{\Cal E}e^{Kt}
\end{equation}
for some positive $C_2$ and $K$ which are independent of $\xi_i\in\Cal B$. In addition, multiplying equation \eqref{differEq} by $(-\Delta)^\frac{1}{2}v$ and using \eqref{5.dif12} and \eqref{5.lip}, analogously to \eqref{2.long}, we get
$$
\int_0^t\|v(s)\|^2_{H^{3/2}(\Omega)}\,ds\le C_2\|\xi_v(0)\|^2_{\Cal E}e^{Kt}
$$
and, therefore,
\begin{equation}\label{5.intsm}
\int_0^t\|\xi_v(s)\|^2_{\Cal E^{1/2}}\,ds\le C_2\|\xi_v(0)\|^2_{\Cal E}e^{Kt}.
\end{equation}
Finally, multiplying equation \eqref{differEq} by $t(-\Dx)^{1/2}\Dt v$ and using again \eqref{5.dif12}, we derive
$$
\frac d{dt}(t\|\xi_v(t)\|^2_{\Cal E^{1/2}})\le  C_3\|\xi_v(t)\|^2_{\Cal E^{1/2}}.
$$
Integrating this inequality in time and using \eqref{5.intsm}, we have
$$
t\|\xi_v(t)\|^2_{\Cal E^{1/2}}\le C_4\|\xi_v(0)\|^2_{\Cal E}e^{Kt}
$$
and the lemma is proved.
\end{proof}
We are now ready to prove the theorem. Indeed, since the embedding $\Cal E_{1/2}\subset\Cal E$ is compact, Lemma \ref{Lem5.difsmo} gives the existence of an exponential attractor $\Cal M_d$ for the discrete semigroup generated by the map $S=S(1)$ on $\Cal B$, see \cite{EMZ00}. The exponential attractor $\Cal M$ for the continuous semigroup can be then obtained via the standard formula
\begin{equation}\label{5.contexp}
\Cal M:=\cup_{t\in[0,1]}S(t)\Cal M_d.
\end{equation}
To guarantee that this set has finite fractal dimension (the other properties of the exponential attractor follow immediately from the fact that $\Cal M_d$ is an exponential attractor for discrete semigroup), it remains to check that the map $(t,\xi)\to S(t)\xi$ is uniformly Lipschitz continuous on $[0,1]\times\Cal B$. The uniform Lipschitz continuity with respect to $\xi$ is guaranteed by \eqref{5.lip} and the Lipschitz continuity
in time follows from the fact that $\xi_{\Dt u}(t)$ is uniformly bounded in $\Cal E$ for any trajectory $\xi_u(t)$ starting from $\xi_u(0)\in\Cal B$, see estimate \eqref{4.e1dis}. Thus, \eqref{5.contexp} is indeed the desired exponential attractor and the theorem is proved.
\end{proof}

To conclude this section, we consider the case when $\alpha=0$ in \eqref{OriginalSys}. In the case of Dirichlet boundary conditions, it does not change anything since, due to the Poincare inequality,
$$
\|\Dt u\|_{H^{1/2}_\Delta}^2\le C\|(-\Dx)^{1/4}\Dt u\|^2_{L^2},
$$
and the term $\gamma(-\Dx)^{1/2}\Dt u$ is enough for energy dissipation. Thus, in that case, all results obtained above remain true for $\alpha=0$ as well.
\par
The case of periodic boundary conditions is more delicate. Indeed, in that case the dissipation vanishes at the spatially homogeneous mode and the dissipative estimate \eqref{2.dis} can be a priori lost. Moreover, it is indeed lost in two elementary cases. First is the case where the nonlinearity $g$ is spatially homogeneous: $g\equiv const$. Then equation \eqref{OriginalSys} possesses spatially homogeneous solutions $\bar u(t,x):=\bar u(t)$ and they solve the ODE
\begin{equation}\label{5.conserv}
\frac{d^2}{dt^2}\bar u(t)+f(\bar u(t))=g
\end{equation}
which is clearly {\it not} dissipative and does not possess a global attractor.
\par
The second one is the case when $f(u):=Lu$, $L>0$, is {\it linear}. In that case, the spatial average $\bar u(t):=\int_\Omega u(t,x)\,dx$ of the solution \eqref{OriginalSys} satisfies equation \eqref{5.conserv} with the spatial average $\bar g$ of the external force $g$ in the right-hand side (instead of $g$), thus the dissipation is again lost. As the next proposition shows, \eqref{OriginalSys} will be nevertheless dissipative in other cases.

\begin{prop}\label{Prop5.Lyap} Let $\alpha=0$, $\Omega=\Bbb T^3$ and the rest of conditions of Theorem \ref{Th2.existence} be satisfied.
Let, in addition, the right-hand side $g$ be not a constant identically and the graph of nonlinearity $f$ do not contain flat segments, i.e., for any $a\in\R$, the set $(f')^{-1}(a)$ be nowhere dense in $\R$.
 Then, the energy functional \eqref{2.energy} is a global Lyapunov function for the solution semigroup $S(t)$ generated by \eqref{OriginalSys} in the energy space $\Cal E$.
\end{prop}
\begin{proof} Indeed, due to \eqref{2.enint}, the energy functional is non-increasing along the trajectories of \eqref{OriginalSys}. Thus, we only need to check that the equality
\begin{equation}\label{5.eqen}
E(u(T),\Dt u(T))=E(u(0),\Dt u(0)),
\end{equation}
for some solution $u(t)$ and some $T>0$, implies that $u(t)\equiv u_0$ is an equilibrium. To prove this fact, we note that, due to \eqref{2.enint},
$$
\int_0^T\|(-\Dx)^{1/4}\Dt u(t)\|^2_{L^2(\Omega)}\,dt=0.
$$
Thus, $\Dt u(t)$ is spatially homogeneous, that is $\Dt u(t)$ does not depend on $x$ and therefore, the solution $u(t,x)$ has the form
$$
u(t,x)=\langle u(t)\rangle-\langle u_0\rangle +u_0(x),
$$
where $\langle u(t)\rangle=\frac{1}{|\Omega|}(u(t),1)$ and $u_0$ is our initial data. Differentiating \eqref{OriginalSys} in time, we see that
\begin{equation}\label{5.f}
f'(u(t))\frac d{dt}\langle u(t)\rangle=-\frac {d^3}{dt^3}\langle u(t)\rangle.
\end{equation}
Assume now that $\langle u(t)\rangle$ is not a constant. Then, there exists time $t_0\in(0,T)$ such that $\frac{d}{dt}\langle u(t_0)\rangle\ne0$. From \eqref{5.f}, we  conclude that
$$
f'(\langle u(t_0)\rangle -\langle u_0\rangle +u_0(x))=a:=-\frac {d^3}{dt^3}\langle u(t_0)\rangle/\frac{d}{dt}\langle u(t_0)\rangle
$$
for all $x\in\Bbb T^3$. Moreover, due to the smoothing property, $u_0(x)\in H^2(\Bbb T^3)\subset C(\Bbb T^3)$. Since $(f')^{-1}(a)$ is nowhere dense, we conclude that $u_0(x)\equiv const$. But then $g(x)$ also must be a constant which contradicts the assumptions of the proposition. This contradiction proves that $\Dt u(t)\equiv0$, so the energy functional is a global Lyapunov function of \eqref{OriginalSys}. Proposition \ref{Prop5.Lyap} is proved.
\end{proof}
\begin{remark} The existence global Lyapunov function together with the evident fact that the set of equilibria is bounded in $\Cal E$ and with the asymptotic compactness (which is also immediate in our case due to the smoothing property) implies the dissipativity and the existence of a global attractor, see e.g., \cite{hale} for the details. Thus, under the assumptions of Proposition \ref{Prop5.Lyap}, we a posteriori have the dissipative estimate \eqref{2.dis} as well as the global and exponential attractors existence. However, in contrast to the case of $\alpha>0$, we do not know how to obtain \eqref{2.dis} directly from the energy-type estimates.
\end{remark}

\section{The case of Dirichlet boundary conditions}\label{s6}

In that section, we verify that the extra regularity \eqref{2.stest} is available for energy solutions in the case of Dirichlet boundary conditions as well. In order to avoid the technicalities, we pose slightly stronger than \eqref{1.f} conditions on the nonlinearity $f$, namely, we assume that $f$ satisfies the following conditions:
\begin{equation}\label{6.2}
\begin{cases}
1.\  f\in C^1(\R,\R),\\
2.\ f'(u)\ge -C+\kappa|u|^4,\ \ |f'(u)|\le C(1+|u|^4),\\
3.\ f(-u)=f(u),
\end{cases}
\end{equation}
for some positive constants $C$ and $\kappa$. Note that, in contrast to \eqref{1.f}, assumptions \eqref{6.2} exclude the non-linearities $f$ with subcritical (less than quintic) growth rate. However, it does not look as a big restriction since the subcritical case is much easier and the desired extra regularity of energy solutions is  straightforward there.
\par
The following analogue of Theorem \ref{Th2.reg} is the main result of the section.

\begin{theorem}\label{Th6.2} Let the problem \eqref{OriginalSys} be equipped by Dirichlet boundary conditions, the non-linearity $f$ satisfy assumptions \eqref{6.2}, $\Omega$ be a smooth bounded domain, $\alpha,\gamma>0$, $g\in L^2(\Omega)$ and let $u(t)$ be a weak solution of problem \eqref{OriginalSys}. Then, $u\in L^2([0,T],H^{3/2}(\Omega))$ and the following estimate holds:
\begin{multline}\label{6.5}
\|u\|_{L^2([t,t+1], H^{3/2}(\Omega))}\le\\\le C\(1+\|\xi_u\|_{L^\infty([t,t+1],\Cal E)}+\|\Dt u\|_{L^2([0,T],H^{1/2}_{\Delta})}+\|g\|_{H^{-1/2}_{\Delta}}\)^3
\end{multline}
for some positive constant $C$ which is independent of $t$ and $u$.
\end{theorem}
\begin{proof}

We first rewrite equation \eqref{OriginalSys} as follows
\begin{equation}\label{6.13}
\Dt^2 u-\Dx(u+v)+f(u)=0,
\end{equation}
where $v(t):=\gamma(-\Dx)^{-1/2}\Dt u+\alpha(-\Dx)^{-1}\Dt u -(-\Dx)^{-1}g$. Then, due to estimate \eqref{2.dis}
\begin{multline}\label{6.14}
\|v\|^2_{L^2([t,t+1],H^{3/2}(\Omega))}\le\\\le C\(\|\xi_u\|_{L^\infty([t,t+1],\Cal E)}^2+\|\Dt u\|^2_{L^2([0,T],H^{1/2}_{\Delta})}+\|g\|_{H^{-1/2}_{\Delta}}^2\).
\end{multline}
Applying the extension operator $\ext$ to both sides of \eqref{6.13} and using \eqref{6.10} together with the fact that $f(u)$ is odd, we have
\begin{equation}\label{6.15}
\Dt^2\tilde u-\Dx(\tilde u+\tilde v)+f(\tilde u)=\tilde h,
\end{equation}
where $\tilde u:=\ext(u)$ and $\tilde v:=\ext(v)$ (see Appendix for the definition and properties of the operator $\ext$)  and
$$
\tilde h:=\sum_{i,j=1}^3\partial_{x_i}(a_{ij}(x)\partial_{x_j}(\tilde u+\tilde v))+\sum_{i=1}^3b_i(x)\partial_{x_i}(\tilde u+\tilde v),
$$
where $a_{ij}(x)$ and $b_i(x)$ are the same as in Lemma \ref{Lem6.4}, see Appendix.
\par
Moreover, due to Lemma \ref{Lem6.3} and Corollary \ref{Cor6.5} together with the growth restrictions on $f$, all terms in \eqref{6.15} are well defined as elements of $L^\infty(0,T;H^{-1}(\Omega_\delta))$.
Thus, we have extended equation \eqref{OriginalSys} initially defined in $\Omega$ to equation \eqref{6.15} which is defined in a larger domain  $\Omega_\delta$. As the next step,  we extend this equation to the whole space $\R^3$ by introducing the cut-off function $\psi(x)=\psi_\eb(x)$ such that
\begin{equation}\label{6.16}
\psi(x)=1,\ \ x\in\Omega_{\eb/2}\ \ \text{and}\ \ \psi(x)=0,\ \ x\notin\Omega_{\eb},
\end{equation}
where $\eb\ll\delta$ is a small parameter which will be fixed below and setting $\bar u=\psi\tilde u$, $\bar v=\psi\tilde v$. Then, these functions satisfy
\begin{equation}\label{6.17}
\Dt^2 \bar u-\Dx(\bar u+\bar v)+\psi f(\tilde u)=\bar h
\end{equation}
with
$$
\bar h:=\sum_{i,j=1}^3\partial_{x_i}(a_{ij}(x)\partial_{x_j}(\bar u+\bar v))+\sum_{i=1}^3\bar b_i(x)\partial_{x_i}(\tilde u+\tilde v)+\bar c(x)(\tilde u+\tilde v)
$$
for some $\bar b_i,\bar c\in L^\infty(\R^n)$  with the support in $\Omega_\eb$. Then, analogously to the space-periodic case,  see Lemma \ref{Lem1.nice}, we introduce the inner product
\begin{multline}\label{6.18}
[U,V]:=\int_{h\in\R^3}\int_{x\in\R^3}\frac{(U(x+h)-U(x))(V(x+h)-V(x))}{|h|^4}\,dx\,dh=\\=c(U,(-\Dx)^{1/2}_{\R^3}V)_{R^3}=
c((-\Dx)^{1/4}_{\R^3}U,(-\Dx)^{1/4}_{\R^3}V)_{\R^3},
\end{multline}
where $(U,V)_{\R^3}$ and $(-\Dx)_{\R^3}$ are the inner product and the Laplacian in the whole space respectively.
Then, obviously,
\begin{equation}\label{6.19}
\begin{cases}
1.\ \ |[U,V]|\le C\|U\|_{L^2(\R^3)}\|V\|_{H^1(\R^3)},\\
2.\ \ |[U,V]|\le C\|U\|_{H^{1/2}(\R^3)}\|V\|_{H^{1/2}(\R^3)},\\
3. \ \ \|U\|_{H^{1/2}(\R^3)}^2\sim \|U\|^2_{L^2(\R^3)}+[U,U].
\end{cases}
\end{equation}
For simplicity, we restrict ourselves to the formal derivation of estimate \eqref{6.5} which can be justified exactly as in the proof of Theorem \ref{Th2.reg} (important that all terms in \eqref{6.17} belong to $H^{-1}(\R^3))$ and, therefore, the approximated inner product $[\cdot,\cdot]_{1/2,\eb}$, see \eqref{1.es}, of the equation with $\bar u\in H^1(\R^3)$ is well-defined).
\par
As in the periodic case, we take inner product \eqref{6.18} of equation \eqref{6.17} with $\bar u$ and integrate over time interval $[t,t+1]$. After integration by parts this gives
\begin{multline}\label{6.20}
[\Dt \bar u(t+1),\bar u(t+1)]-[\Dt\bar u(t),\bar u(t)]-\int_{t}^{t+1}[\Dt\bar u(s),\Dt\bar u(s)]\,ds+\\+\int_t^{t+1}[\Nx\bar u(s),\Nx \bar u(s)]\,ds+\int_t^{t+1}[\Nx \bar v(s),\Nx\bar u(s)]\,ds+\\+\int_t^{t+1}[\psi f(\tilde u(s)),\bar u(s)]\,ds=\int_t^{t+1}[\bar h(s),\bar u(s)]\,ds.
\end{multline}
Thus, we only need to estimate the terms in \eqref{6.20}. First, due to Lemma \ref{Lem6.3},
\begin{multline}\label{6.21}
\|\Dt \bar u\|_{L^2(\R^3)}\le C\|\Dt u\|_{L^2(\Omega)},\\ \|\bar u\|_{H^1(\R^3)}\le C\|u\|_{H^1(\Omega)},\ \|\Dt \bar u\|_{H^{1/2}(\R^3)}\le C\|\Dt u\|_{H^{1/2}_{\Delta}}
\end{multline}
and, therefore, due to \eqref{6.18} and \eqref{2.dis}, first 3 terms in \eqref{6.20} is controlled by the energy norm of the solution $u$. Second, according to Lemma \ref{Lem6.3}, we also have
\begin{equation}\label{6.22}
\|\bar v\|_{H^{3/2}(\R^3)}\le C\|v\|_{H^{3/2}_{\Delta}}\le C(\|\Dt u\|_{H^{1/2}_{\Delta}}+\|g\|_{H^{-1/2}_{\Delta}})
\end{equation}
and, therefore, the 5th term in the left-hand side of \eqref{6.20} can be estimated as follows
\begin{multline}\label{6.23}
[\Nx\bar v,\Nx\bar u]\le \frac12[\Nx\bar u,\Nx\bar u]+\\+\frac12[\Nx\bar v,\Nx\bar v]\le \frac12[\Nx\bar u,\Nx\bar u]+C\|\Nx\bar v\|^2_{H^{1/2}(\R^3)}
\end{multline}
and since the $H^{1/2}(\R^3)$-norm of the gradient is controlled by the $H^{3/2}(\R^3)$-norm, the 5th term is also controlled by the 4th one and the energy norm of the solution $u(t)$.
\par
As the third step, we estimate the right-hand side of \eqref{6.20}. To this end, we note that all terms in $\bar h$ which does not contain second derivatives in space can be straightforwardly controlled by the energy norm of the solution using the first estimate of \eqref{6.19} and \eqref{6.21}.
Moreover, the terms which contain the second derivatives of $\bar v$ can be estimated analogously to the 5th term using the fact that $a_{ij}\in W^{1,\infty}$ are the multipliers in $H^{1/2}(\R^3)$, see also the estimate of the next term below. Thus, we only need to estimate the terms
\begin{multline*}
[a_{ij}\partial_{x_i}\bar u,\partial_{x_i}\bar u]=I_1+I_2:=\\=\int_{h\in\R^3}\int_{x\in\R^3}a_{ij}(x)\frac{(\partial_{x_i}\bar u(x+h)-\partial_{x_i}\bar u(x))(\partial_{x_j}\bar u(x+h)-\partial_{x_j}\bar u(x))}{|h|^4}\,dx\,dh+\\+\int_{h\in\R^3}\int_{x\in\R^3}\partial_{x_i}\bar u(x+h)\frac{(a_{ij}(x+h)-a_{ij}(x))(\partial_{x_j}\bar u(x+h)-\partial_{x_j}\bar u(x))}{|h|^4}\,dx\,dh.
\end{multline*}
To estimate $I_2$, we remind that $a_{ij}\in W^{1,\infty}(\R^n)$ and the integration in $h$ can be done for $|h|\le K$  only 
($K>0$ is fixed, the remaining integral over $|h|>K$ is controlled by the $L^2$-norm of $\Nx \bar u$), consequently, by the Cauchy-Schwartz inequality
\begin{multline*}
I_2\le \int_{|h|\le K}\int_{x\in\R^3}|\partial_{x_i}\bar u(x+h)|\frac{|\partial_{x_j}\bar u(x+h)-\partial_{x_j}\bar u(x)|}{|h|^3}\,dx\,dh+C_K\|\bar u\|_{H^1(\R^3)}^2\le\\
\le [\partial_{x_j}\bar u,\partial_{x_j}\bar u]^{1/2}\(\int_{|h|\le K}\int_{x\in\R^3}\frac{|\partial_{x_i}\bar u(x+h)|^2}{|h|^2}\,dx\,dh\)^{1/2}+C_K\|\bar u\|^2_{H^1(\R^3)}\le\\\le
C[\partial_{x_j}\bar u,\partial_{x_j}\bar u]^{1/2}\|\bar u\|_{H^1(\R^3)}+C_K\|\bar u\|^2_{H^1(\R^3)}.
\end{multline*}
Thus, this term is controlled by the 4th term of the left-hand side of \eqref{6.20} and the energy norm. To estimate $I_1$, we first note that, obviously,
\begin{multline*}
I_1\le \int_{|h|<\eb}\int_{x\in\R^3}a_{ij}(x)\frac{(\partial_{x_i}\bar u(x+h)-\partial_{x_i}\bar u(x))(\partial_{x_j}\bar u(x+h)-\partial_{x_j}\bar u(x))}{|h|^4}\,dx\,dh+\\+C_\eb\|\bar u\|^2_{H^1(\R^3)}.
\end{multline*}
To estimate the first integral, we remind that $a_{ij}(x)=0$ if $x\in\Omega$ and both $\Nx\bar u(x+h)$ and $\Nx\bar u(x)$ equal to zero if $x\notin\Omega_{2\eb}$ (here we used the definition \eqref{6.16} of the cut-off function $\psi$ and the restriction $|h|<\eb$). Thus, the integrand is non-zero when $x$ belongs to the $2\eb$-layer $\Omega_{2\eb}\backslash\Omega$ only. Since $a_{ij}\in W^{1,\infty}$ and $a_{ij}\big|_{\partial\Omega}=0$, they are of order $\eb$ in that layer. Thus,
$$
I_1\le C\eb[\Nx \bar u,\Nx\bar u]+C_\eb\|\bar u\|^2_{H^1(\R^3)}.
$$
Combining the obtained estimates, we see that
\begin{multline}\label{6.24}
\int_t^{t+1}[\bar h(s),\bar u(s)]\,ds\le C\eb\int_t^{t+1}[\Nx\bar u(s),\Nx\bar u(s)]\,ds+\\+C_\eb\int_t^{t+1}(\|\bar u(s)\|^2_{H^1}+\|\bar v(s)\|^2_{H^{3/2}})\,ds.
\end{multline}
Thus, all the terms in \eqref{6.20} except of the one containing the nonlinearity are estimated. Inserting the obtained estimates to the equality \eqref{6.20} and fixing $\eb$ to be small enough, we end up with the estimate
\begin{multline}\label{6.25}
\int_t^{t+1}[\Nx \bar u(s),\Nx\bar u(s)]\,ds+\int_t^{t+1}[\psi f(\tilde u(s)),\psi\tilde u(s)]\,ds\le \\ \le
C\(\|\xi_u\|^2_{L^\infty([t,t+1],\Cal E)}+\|\Dt u\|^2_{L^2([t,t+1],H^{1/2}_{\Delta})}+\|g\|^2_{H^{-1/2}_{\Delta}}\).
\end{multline}
Therefore, we only need to estimate the second term in the left hand side of \eqref{6.25}. Note that, by adding the linear term to $f$, we may assume without loss of generality that $f'(u)\ge \kappa (1+u^4$) and, consequently,
\begin{equation}\label{6.26}
(f(a)-f(b))(a-b)\ge \beta(1+|a|+|b|)^4(a-b)^2.
\end{equation}
On the other hand,
\begin{equation}\label{6.27}
|(f(a)-f(b)|\le C(1+|a|+|b|)^4|a-b|,\ \ |f(a)|\le C(1+|a|^5).
\end{equation}
Using these formulas and the fact that $\psi$ is smooth, we get
\begin{multline*}
\biggl(\psi(x+h)f(\tilde u(x+h))-\psi(x)f(\tilde u(x))\biggr)\biggl(\psi(x+h)\tilde u(x+h)-\psi(x)\tilde u(x)\biggr)=\\=
\biggl(\psi(x+h)(f(\tilde u(x+h))-f(\tilde u(x)))+f(\tilde u(x))(\psi(x+h)-\psi(x))\biggr)\times\\\times\biggl(\psi(x+h)(\tilde u(x+h)-\tilde u(x))+\tilde u(x)(\psi(x+h)-\psi(x))\biggr)\ge\\\ge
\beta\psi^2(x+h)\biggl(1+|\tilde u(x+h)|+|\tilde u(x)|\biggr)^4\biggl(\tilde u(x+h)-\tilde u(x)\biggr)^2-\\-C |h|\psi(x+h)\biggl(|\tilde u(x)||f(\tilde u(x+h))-f(\tilde u(x))|+|f(\tilde u(x))||\tilde u(x+h)-\tilde u(x)|\biggr)-\\-C|h|^2|f(\tilde u(x))\tilde u(x)|\ge\\\ge
\beta\psi^2(x+h)\biggl(1+|\tilde u(x+h)|+|\tilde u(x)|\biggr)^4\biggl(\tilde u(x+h)-\tilde u(x)\biggr)^2-\\-
C|h|\psi(x+h)\biggl(1+|\tilde u(x+h)|+|\tilde u(x)|\biggr)^5\biggl|\tilde u(x+h)-\tilde u(x)\biggr|-\\-
C|h|^2\biggr(1+|\tilde u(x+h)|+|\tilde u(x)|\biggr)^6\ge\\\ge
\frac{\beta}2\psi^2(x+h)\biggl(1+|\tilde u(x+h)|+|\tilde u(x)|\biggr)^4\biggl(\tilde u(x+h)-\tilde u(x)\biggr)^2-\\-
C_1|h|^2\biggl(1+|\tilde u(x+h)|^6+|\tilde u(x)|^6\biggr).
\end{multline*}
That estimates give
$$
[\psi f(\tilde u),\psi\tilde u]\ge -C_2(1+\|\tilde u\|^6_{L^6(\Omega_\delta)})\ge -C_3(1+\|\tilde u\|_{H^1(\Omega_\delta)}^6)\ge -C_4(1+\|u\|^6_{H^1(\Omega)}).
$$
Inserting this estimate to \thetag{25}, we end up with
\begin{multline}\label{6.28}
\int_t^{t+1}[\Nx \bar u(s),\Nx\bar u(s)]\,ds\le \\ \le
C\(\|\xi_u\|^2_{L^\infty([t,t+1],\Cal E)}+\|\Dt u\|^2_{L^2([t,t+1],H^{1/2}_{\Delta})}+\|g\|^2_{H^{-1/2}_{\Delta}}\)^3.
\end{multline}
Estimate \eqref{6.5} is an immediate corollary of this estimate and the obvious fact that
$$
\|u\|^2_{H^{3/2}(\Omega)}\le \|u\|^2_{H^1(\Omega)}+\int_{\Omega}\int_{\Omega}\frac{|\Nx u(x)-\Nx u(y)|^2}{|x-y|^4}\,dx\,dy\le \|u\|^2_{H^1(\Omega)}+[\Nx\bar u,\Nx\bar u].
$$
Thus, Theorem \ref{Th6.2} is proved.
\end{proof}
\begin{remark} As it has been shown before, the extra regularity of energy solutions established in Theorem \ref{Th6.2} is enough to verify the well-posedness of energy solutions for problem \eqref{OriginalSys} as well as their dissipativity, smoothing property and the existence of finite-dimensional global and exponential attractors. Thus, these results are proved for the case of Dirichlet boundary conditions under the assumptions of Theorem \ref{6.2}.
\end{remark}
\section{Appendix. Properties of the extension operator}\label{s7}

The aim of this appendix is to  define and study the odd extension operator for functions defined in a smooth bounded domain of $\R^3$ which vanish at the boundary $\partial\Omega$. This operator is a crucial technical tool for proving the additional regularity of energy solutions for the case of Dirichlet boundary conditions.
\par
 Namely, since $\Omega$ is smooth, any  point $x$ in the small $\delta$ neighborhood $\Cal O_\delta(\partial\Omega)$ of $\partial\Omega$ can be presented in a unique way the form
\begin{equation}\label{6.6}
x=x'+s\vec n,
\end{equation}
where $x'\in\partial\Omega$, $\vec n$ is a normal to $\partial\Omega$ at $x'$ and $s\in(-\delta,\delta)$. Thus, \eqref{6.6} realizes a diffeomorphism
of $\Cal O_\delta(\partial\omega)$ and $\partial\Omega\times(-\delta,\delta)$ and we will treat the pair $(x',s)$ as coordinates in the neighborhood $\Cal O_\delta(\partial\Omega)$ of the boundary.
\par
In that coordinates the reflection $R_\Omega$ with respect to the boundary reads
\begin{equation}\label{6.7}
R_\Omega:\Cal O_\delta(\partial\Omega)\to\Cal O_\delta(\partial\Omega),\ \ (x',s)\to(x',-s)
\end{equation}
which corresponds to the $C^\infty$-map $y=R_\Omega(x)$ in the initial coordinates. Note also that $R_\Omega(x)$ maps points which are {\it inside} of $\Omega$ to the points {\it outside} of $\Omega$ and
$$
R_{\Omega}(R_{\Omega}(x))\equiv x.
$$
 Moreover, for any $x\in\partial\Omega$, that the derivative $R'_\Omega(x)$ is the usual (linear) {\it reflection} with respect to the tangent plane to $\partial\Omega$ at $x$.
\par
We define the desired extension operator $\ext$ as follows:
\begin{equation}\label{6.8}
\ext(u)(x):=\begin{cases} u(x),\ \ &x\in\Omega,\\ -u(R_\Omega(x)),\ \ &x\in \Cal O_\delta(\Omega)\backslash\Omega.\end{cases}
\end{equation}
The next lemma shows that $\ext$ defines indeed a proper extension of functions $H^s_{\Delta}$, $-1\le s\le 2$ outside of $\Omega$.

\begin{lemma}\label{Lem6.3} Let $\Omega$ be a smooth domain. Then $\ext$ is a linear continuous operator
\begin{equation}\label{6.9}
\ext: H^s_{\Delta}\to H^s(\Omega_\delta)
\end{equation}
for all $-1\le s\le2$ (here and below, $\Omega_\delta:=\{x\in\R^3, \dist(x,\Omega)<\delta\}$).
\end{lemma}
\begin{proof}  Indeed, let $u\in H^2(\Omega)\cap H^1_0(\Omega)$. Then, since $u\big|_{\partial\Omega}=0$, we have
$$
u\big|_{\partial\Omega}=\ext(u)\big|_{\partial\Omega}=0.
$$
Moreover, since the extension is odd,
$$
\partial_{\vec n} u\big|_{\partial\Omega}=\partial_{\vec n}\ext(u)\big|_{\partial\Omega}.
$$
Therefore, there are no jumps of first derivatives on $\partial\Omega$ and, consequently, no singular
 parts for the second derivatives as well. Thus, $\ext(u)\in H^2(\Omega_\delta)$ and \eqref{6.9} is proved for $s=2$. For $s=1$ and $s=0$ it is evident and for the non-integer exponents $s$ it follows then by interpolation. Thus, \eqref{6.9} is verified for non-negative $s\in[0,2]$. To prove it for $s\in[-1,0)$, we take any $u\in C_0^\infty(\Omega)$ and $\varphi\in C_0^\infty(\Omega_\delta)$ and use the following identity
 \begin{multline*}
 \int_{\Omega_\delta}\ext(u)\varphi\,dx=\int_\Omega u(x)\varphi(x)\,dx-\int_{\Omega_\delta\backslash\Omega}u(R_{\Omega}(x))\varphi(x)\,dx=
 \int_\Omega u(x)\varphi(x)\,dx-\\-\int_{\Omega}u(y)\varphi(R_{\Omega}(y))|\det R'_{\Omega}(y)|\,dy=\\=\int_{\Omega}u(x)(\varphi(x)-|\det R'_{\Omega}(x)|\varphi(R_{\Omega}(x)))\,dx:=\int_\Omega u{\ext}^*(\varphi)\,dx,
 \end{multline*}
 where, by definition, $\varphi(R_{\Omega}(x))\equiv0$ if $x\notin\Cal O_\delta(\partial \Omega)$. This identity shows that $\ext^*(\varphi)\in H^1_0(\Omega)$
 (since $|\det R'_{\Omega}(x)|=1$ when $x\in\partial\Omega$) and
 $$
 \|\ext^*(\varphi)\|_{H^1_0(\Omega)}\le C\|\varphi\|_{H^1_0(\Omega_\delta)}.
 $$
 Thus, by density arguments, $\ext$ is a linear continuous operator from $H^{-1}(\Omega)$ to $H^{-1}(\Omega_\delta)$. For non-integer $s\in(-1,0)$, \eqref{6.9} follows again by interpolation and Lemma \ref{Lem6.3} is proved.
\end{proof}
The next  lemma which gives the expression for the commutator of $\ext$ and the Laplacian is crucial for what follows.

\begin{lemma}\label{Lem6.4} Let $\Omega$ be a smooth domain and let $u\in C^\infty(\overline{\Omega})$ be a smooth function satisfying $u\big|_{\partial\Omega}=0$. Then
\begin{equation}\label{6.10}
\ext(\Dx u)-\Dx(\ext(u))=\sum_{i,j=1}^3\partial_{x_i}(a_{ij}(x)\partial_{x_j}\ext(u))+\sum_{i=1}^3b_i(x)\partial_{x_i}\ext u,
\end{equation}
where $a_{ij}\in W^{1,\infty}(\Omega_\delta)$, $b_i\in L^\infty(\Omega_\delta)$ and
\begin{equation}\label{6.11}
a_{ij}\big|_{\Omega}\equiv0.
\end{equation}
\end{lemma}
\begin{proof}  Indeed, since $R'_\Omega(x)$ is a reflection and the reflections preserve the Laplacian, we have
\begin{multline}\label{6.12}
\Dx (u(R_\Omega(x)))=(\Dx u)(R_\Omega(x))+\sum_{ij}\tilde a_{ij}(x)(\partial_{x_i}\partial_{x_j}u)(R_\Omega(x))+
\\+\sum_{i}\tilde b_i(x)(\partial_{x_i}u)(R_\Omega(x)),
\end{multline}
where $\tilde a_{ij}\big|_{\partial\Omega}=0$ and the assertion of the lemma is an immediate corollary of \eqref{6.12} and the definition of the extension operator $\ext$. Lemma \ref{Lem6.4} is proved.
\end{proof}
\begin{cor}\label{Cor6.5} Formula \eqref{6.10} remains valid if $u\in H^s_\Delta$, $1\le s\le2$.
\end{cor}
Indeed, the assertion of the corollary follows by the standard density arguments from the facts that $C^\infty(\overline{\Omega})\cap H^1_0(\Omega)$ is dense in $H^s_{\Delta}$ and that both right and left hand sides have sense for $u\in H^s_\Delta$ for that values of $s$.

\end{document}